\documentclass{amsart}

\usepackage{amsmath}
\usepackage{amssymb}
\usepackage{verbatim}
\usepackage{enumerate}
\usepackage[all]{xy}

\DeclareMathAlphabet{\mathpzc}{OT1}{pzc}{m}{it}

\numberwithin{equation}{section}
\newtheorem{thm}[equation]{Theorem}
\newtheorem*{thm*}{Theorem}

\newtheorem{prop}[equation]{Proposition}
\newtheorem{defn}[equation]{Definition}
\newtheorem*{remark*}{Remark}

\newtheorem{lemma}[equation]{Lemma}
\newtheorem{remark}[equation]{Remark}
\newtheorem{ex}[equation]{Example}
\newtheorem{coro}[equation]{Corollary}

\newcommand{\exref}[1]{Ex\-am\-ple \ref{#1}}
\newcommand{\thmref}[1]{Theo\-rem \ref{#1}}

\newcommand{\lemref}[1]{Lem\-ma \ref{#1}}
\newcommand{\propref}[1]{Prop\-o\-si\-tion \ref{#1}}
\newcommand{\corref}[1]{Cor\-ol\-lary \ref{#1}}

\newcommand{\remref}[1]{Re\-mark \ref{#1}}

\providecommand{\intersect}{\cap}

\newenvironment{block*}
		{\hspace*{0.05\textwidth}\begin{minipage}[t]{0.9\textwidth}}
		{\end{minipage}\\}

\newlength{\categorywidth}       
\usepackage{calc}
\newlength{\infowidth}          
\newlength{\categorysep}
\newlength{\entrysep}
\DeclareMathOperator{\chr}{char }

\DeclareMathOperator{\rad}{rad }

\DeclareMathOperator{\End}{End }

\DeclareMathOperator{\Sym}{Sym }

\DeclareMathOperator{\Adj}{Adj }

\DeclareMathOperator{\GL}{GL}

\DeclareMathOperator{\GF}{GF}
\DeclareMathOperator{\Gal}{Gal }

\DeclareMathOperator{\Aut}{Aut}

\DeclareMathOperator{\memb}{memb}

\DeclareMathOperator{\rank}{rank }

\newcommand{\Bi}{\mathsf{Bi} }

\begin{document}

\title{Finding central decompositions of $p$-groups}
\author{James B. Wilson}
\address{
	Department of Mathematics\\
	The Ohio State University\\
	Columbus, Ohio 43210\\
}
\email{wilson@math.ohio-state.edu}
\date{\today}
\thanks{This research was supported in part by NSF Grant DMS 0242983.}
\keywords{central products, $p$-groups, Lie rings, bilinear maps, 
$*$-rings, polynomial time, Las Vegas}

\begin{abstract} Polynomial-time
algorithms are given to find a central decomposition of maximum size
for a finite $p$-group of class $2$ and for a nilpotent Lie ring
of class $2$. 
The algorithms use Las Vegas probabilistic
routines to compute the structure of finite $*$-rings and also 
the Las Vegas \textsf{C-MeatAxe}.  When $p$ is small, 
the probabilistic methods can be replaced by deterministic 
polynomial-time algorithms.

The methods introduce new group isomorphism invariants
including new characteristic subgroups.
\end{abstract}

\maketitle

%
%
\section{Introduction}

The main goal of this paper is to prove:
\begin{thm}\label{thm:main}
There are deterministic and Las Vegas polynomial-time algorithms which, 
given a finite $p$-group $P$ of class $2$, return a set $\mathcal{H}$ of subgroups of $P$ where distinct members pairwise commute, and of 
maximum size such that $\mathcal{H}$ generates
$P$ and no proper subset does.
\end{thm}
We call $\mathcal{H}$ a \emph{central decomposition} of $P$ since
$P$ is a central product of the groups in $\mathcal{H}$ (with
centers permitted to overlap haphazardly) \cite[(11.1)]{Aschbacher}.
$P$ is input as a permutation, matrix, or (black-box) polycyclic
group.

\thmref{thm:main} applies a new group isomorphism invariant
for $p$-groups: an associative ring with involution, i.e.: a $*$-ring. 
Central decompositions are a natural 
application of these $*$-ring methods and appear to be undetectable by
conventional $p$-groups methods such as using factors of a characteristic central series.  The $*$-rings convert the commutation structure of a $p$-group into classical questions about ring structure which can be computed using linear algebra.  The ``atoms'' of a central decomposition (\emph{centrally indecomposable} subgroups) have specific associated $*$-rings making them detectable and restricting 
their structure.

\thmref{thm:main} applies broadly, but \emph{special} groups are our 
main focus, specifically, $p$-groups $P$ with elementary abelian Frattini subgroup $\Phi(P)=P'=Z(P)$. These groups have few discernible
characteristic subgroups, so group isomorphism invariants of any kind are helpful.  Despite their name, special groups are diverse, 
comprising at least $p^{2n^3/27-4n^2/9}$ of the at most 
$p^{2n^3/27+O(n^{8/3})}$ groups of order $p^n$ 
\cite[Theorem 2.3]{Higman:enum}, \cite[p. 153]{Sims:enum}.  
While there are $p^{2n^3/27+O(n^2)}$ centrally decomposable
special groups of order $p^n$ (e.g.: $P\times \mathbb{Z}_p$), 
using $*$-rings shows there are $p^{2n^3/27+O(n^2)}$ centrally
indecomposable special groups of order $p^n$ as well \cite{Wilson:indecomp}.  

Using the group isomorphism invariants developed for \thmref{thm:main}, we introduce various other applications such as defining new characteristic and fully invariant subgroups of $p$-groups as well as algorithms to find generators for these subgroups. We also consider the problem of central products of general groups and explain the importance of P. Hall's isoclinism to the study of central decompositions.  The details of these applications as well as useful examples are provided in the closing Sections \ref{sec:closing} and
\ref{sec:general}, and Appendices \ref{app:simples} and \ref{app:Tang}.

Whereas it is customary to use nilpotent Lie rings in order to exploit
the commutation of a $p$-group, this does not seem to be helpful 
for central decompositions.  Indeed, our $*$-rings are nonnilpotent
and can be simple, semisimple, or have large radicals.  
We use the radical and semisimple structure of $*$-rings 
for \thmref{thm:main}, and also for the obvious analogue, \thmref{thm:Lie}, for nilpotent Lie rings of class $2$ (including 
characteristic $0$).

Central products have various irregularities which set them apart 
from the more familiar but special case of direct products.  If 
$\mathcal{H}$ is a central decomposition of $P$ and $H\in\mathcal{H}$, then considering $P/H$ can omit the intricate intersections of the members of $\mathcal{H}-\{H\}$.  Therefore, inductive proofs and greedy algorithms seem impossible with central products.  There can be no Theorem of Krull-Remak-Schmidt type for central products, for example, $D_8\circ D_8\cong Q_8\circ Q_8$ and similar examples for
odd extraspecial $p$-groups \cite[Theorem 5.5.2]{Gor}.  More strikingly, C. Y. Tang \cite[Section 6]{Tang:cent-2} gives a group of order $2^{12}$ which is the central product of two centrally indecomposable subgroups, but also the central product of three centrally indecomposable subgroups (Example \ref{ex:Tang}).
\thmref{thm:main} finds a central decomposition of maximum length 
in one pass, rather than through the gradual refinement of an evolving central decomposition, and so avoids the latter problem.  

\begin{remark}
There can be \emph{any} number of $\Aut P$-orbits of 
 central decompositions as in \thmref{thm:main}, but if $P$ has class $2$ and exponent
$p$, these orbits can be classified using Jordan algebras 
\cite[Theorem 1.1]{Wilson:unique}.
\end{remark}

The algorithms for \thmref{thm:main} perform with roughly the same
asymptotic efficiency as algorithms for modules of a comparable size.
Essential tools for our algorithms include the \textsf{MeatAxe}  
\cite{MeatAxe1,MeatAxe2,C-MeatAxe} and algorithms for rings 
introduced by Ronyai, Friedl, and Ivanyos \cite{Ronyai,Ivanyos:fast-alge}. 

\begin{remark}
The author and P. A. Brooksbank recently revisited
the essential algorithms for $*$-rings introduced in Sections
\ref{sec:adj} and \ref{sec:*-rings} \cite{Brooksbank-Wilson:*-ring,
Brooksbank-Wilson:isom}.  The resulting algorithms make 
greater use of fast module theory methods, improve the complexity 
of those sections, and are implemented for use in \textsf{MAGMA} \cite{Magma}.  Early tests have handled randomized examples for $p$-groups of size $p^{45}$ with rank $36$ and $p=3,5,7,11$, and used roughly five seconds of real-time on a conventional laptop, and examples of size $p^{196}$ with intentionally complex central decompositions took one hour on a laboratory computer with extensive memory; details are included in \cite{Brooksbank-Wilson:*-ring,Brooksbank-Wilson:isom}.
\end{remark}

\subsection{Survey of the paper}
Section \ref{sec:background} consists of background.

In Section \ref{sec:groups}, our algorithm passes from $P$ to the 
bilinear map $b:P/Z(P)\times P/Z(P)\to P'$ of commutation in $P$.  
It is shown that central 
decompositions of $P$ correspond to orthogonal decompositions of $b$
(\propref{prop:cent-perp} and \thmref{thm:cent-indecomp}).  To find a 
fully refined orthogonal decomposition of $b$, the ring of adjoints
of $b$ is computed.  This is a natural $*$-ring.
Continuing the translation of the problem, orthogonal decompositions
are related to self-adjoint idempotents of $\Adj(b)$ 
(\corref{coro:perp-idemp}). 
These translations occupy Section \ref{sec:adj}. 

\begin{remark}
As suggested above, when $b$ is a bilinear \emph{map}
(rather than a form) the ring of adjoints can be far from simple and can have a rich structure of radicals and semisimple factors.  Examples can be 
constructed to demonstrate this structure occurs within our application 
to $p$-groups, even for $p$-groups of small order \cite[Section 7]{Wilson:unique}.
\end{remark}

In Section \ref{sec:*-rings}, we begin the process of constructing
self-adjoint idempotents by using the semisimple and radical structure of $\Adj(b)$.  This structure can be computed efficiently by reducing to rings of characteristic $p$ and applying the algorithms of Ronyai, Friedl, and Ivanyos for finite $\mathbb{Z}_p$-algebras 
\cite{Ronyai,Ivanyos:fast-alge,Ronyai-Ivanyos:tapas}.  This stage uses
Las Vegas polynomial-time algorithms for factoring polynomials over
finite fields of characteristic $p$, such as the methods of Berlekamp
or Cantor-Zassenhaus
\cite[Chapter 14]{vzG}.  However, 
for a deterministic algorithm (for small $p$), Las Vegas
algorithms can be avoided.

Section \ref{sec:final} includes the proof of \thmref{thm:main}
by first finding a orthogonal decomposition of $b$ of maximum possible size and converting this to a central decomposition of $P$ of maximum possible possible size.

Section \ref{sec:closing} creates the analogue of \thmref{thm:main} for nilpotent Lie rings of class $2$, introduces the four families of centrally indecomposable $p$-groups, and presents new characteristic subgroups which are easily identified using $\Adj(b)$.

Section \ref{sec:general} shows how the nonabelian members of a central decompositions are preserved by group isoclinisms of any group, not only finite $p$-groups.  There a conjecture is given concerning the uniqueness of central decompositions of maximum possible size.  Then the r\^{o}le of adjoints is then expanded to central products of general groups is explained.

The appendices give examples which demonstrate that the cases considered in Section \ref{sec:*-rings} do occur in the context of finite $p$-groups.  We also provide an alternative proof of the example of C.Y. Tang \cite[Section 6]{Tang:cent-2} using the methods of \thmref{thm:main}.  Our proof extends the example to an infinite expanding family of examples.

%
%
\section{Background}\label{sec:background}
Throughout this work we assume $p$ is a prime.  Unless otherwise obvious, all our groups, rings, modules, and algebras are finite.  
All our associative rings are unital.  We express abelian groups 
additively.

We use $A\sqcup B$ for the disjoint union of sets $A$ and
$B$, and $A-B$ for the complement of $A\intersect B$ in $A$.
For details on computational complexity and rigorous treatments
of polynomial-time and Las Vegas algorithms see \cite[Chapter I]{Seress:book}.

For a $p$-group $P$, we let $P'=[P,P]$ denote the derived
subgroup of $P$, $Z(P)$ the center of $P$, and $\Phi(P)$ the
Frattini subgroup of $P$.

We have need in various places to apply homomorphisms and isomorphisms
between finite abelian $p$-groups, rings, and algebras.  We say a homomorphism
is \emph{effective} when it can be evaluated efficiently -- for instance with
the same cost as matrix multiplication -- and a coset representative for
the preimage of an element in the codomain can also be found efficiently.  
This means that effective isomorphisms are easily evaluated and inverted
on any desired element.  

\subsection{Central products and central decompositions}

The term \emph{central product} was invented by P. Hall to describe
a specific type of amalgamated product especially common when constructing
$p$-groups  \cite[Section 3.2]{Hall:fin-by-nil}.  Specifically,
a \emph{central product} over a set $\mathcal{H}$ of groups is 
an epimorphism $\varphi:\prod_{H\in\mathcal{H}} H \to G$ such that 
$H\intersect \ker \varphi=1$ for all $H\in\mathcal{H}$ 
\cite[(11.1)]{Aschbacher}.  The problem with that definition is that it allows any epimorphism, for instance, $\mathbb{Z}_p^n\to \mathbb{Z}_p$ so that $\mathbb{Z}_p$ is a central product of an arbitrary number of groups.  To avoid this obvious degeneracy, we consider only central products which have the added constraint: $\langle \mathcal{J}\rangle\varphi =G$ for $\mathcal{J}\subseteq \mathcal{H}$ implies $\mathcal{J}=\mathcal{H}$.  All other central products will be known as \emph{degenerate} so that by default central products are nondegenerate.

A \emph{central decomposition} is a set $\mathcal{H}$ of subgroups of $P$ which generates $P$, no proper subset does, and distinct members commute.  Note that $1$ is never in a central decomposition.  When $\{P\}$ is the only central decomposition of $P$, then $P$ is \emph{centrally indecomposable}.  A central decomposition is 
\emph{fully refined} when its members are centrally indecomposable.
If $\mathcal{H}$ is a central decomposition of $P$, then the direct
product $\prod_{H\in\mathcal{H}} H$ maps homomorphically onto $P$ via 
$(x_H)_{H\in\mathcal{H}}\mapsto \prod_{H\in\mathcal{H}} x_H$, and
the kernel of the map intersect each $H\in\mathcal{H}$ trivially.  
Thus, central decompositions give rise to central products, and
vice-versa; compare \cite[(11.1)]{Aschbacher}.

\begin{remark}
These definitions are not sufficient to guarantee that
a central decomposition of an abelian group is a direct product
(e.g.: $\{\langle (1,0)\rangle,\langle (1,1)\rangle\}$ is 
a central decomposition of $\mathbb{Z}_{p^2}\times \mathbb{Z}_p$ 
but not a direct decomposition).  Yet, all fully refined central decompositions of an abelian group have size equal to the rank 
of the group and our algorithms for \thmref{thm:main} make an
effort to return direct factors when possible.
\end{remark}

\subsection{Representing groups for computation}\label{sec:grp}

We assume throughout that $P$ is a finite $p$-group of class $2$
(i.e.: $P'\leq Z(P)$) for a known prime $p$.  Groups and subgroups will be specified with generators; so, $P=\langle S\rangle$.  We will not consider the specific representation of $P$, but assume only that it can be input with $O(|S|n)$ bits of data (ex: $n=|\Omega|$ if $G$ acts faithfully on $\Omega$ and $n=d^2\log q$ if $P\leq \GL(d,q)$) 
and that there are polynomial-time, in $n$, algorithms which: multiply, invert, and test equality of elements in $P$; and also test membership, i.e.: given $g\in P$ and $T\subseteq P$, determine if $g\in\langle T\rangle$. The first three problems have standard $O(n^{2})$-time algorithms (or better).  However, the membership-test algorithms are considerably more involved, see \cite[Section 3.1]{HoltEO}, \cite[Chapters 3-4]{Seress:book}, and \cite[Theorem 3.2]{Luks:mat}.

\begin{remark}
Polycyclic groups can also be used as input; however, there are no 
known polynomial-time algorithm to multiply with such groups 
\cite[p. 670]{LG-Soicher:collector}.   Hence, \thmref{thm:main} 
treats polycyclic group inputs as ``black-box'' groups so that
polynomial-time refers to a the total number of group multiplications and membership tests.
\end{remark}

The assumptions on $P$ given thus far lead to deterministic polynomial-time algorithms which: find $|\langle T\rangle|$ for any $T\subseteq P$, find generators for the normal closure $\langle T^G\rangle$ of $T\subseteq P$; find generators for $P'$, and find generators for $Z(P)$ \cite[Section 3.3]{HoltEO}.  These are the additional algorithms we assume for our $p$-groups.  

We will use the following in the timing of our algorithms:
\begin{enumerate}[(i)]
\item $\memb(P)$ -- the time to perform membership test in $P$,
\item $\rank P$ -- the rank of $P$, i.e. $\log_p [P:\Phi(P)]$,
\item $\exp(P)$-- the exponent of $P$, i.e. the smallest $p^e$ such that
$P^{p^e}=1$.
\end{enumerate}
Both $P'$ and $Z(P)$ can be computed once at the start of our algorithms, and will not contribute to the overall complexity.  We store any relevant elements of our groups as words (straight-line-programs) in the original generating set of $P$.  We define homomorphisms by the images of the generators and therefore pulling back elements of the images can be done by pulling back words in the appropriate generating sets.

\subsection{Central products and discrete logs}

Suppose that $P\leq \GL(d,q)$ with $p>d$ and $(p,q)=1$.  This is enough to require that $P$ embed in $A:=\GF(q^{e_1})^\times \times\cdots \times\GF(q^{e_s})^\times$, and so $P$ is abelian.  The centrally indecomposable abelian groups are cyclic of prime power order.  However, to determine that a subgroup of $A$ is cyclic appears to be a very difficult number theory problem in general solved (in non-polynomial-time) by discrete logs \cite[Section 7.1]{HoltEO}.

For \thmref{thm:main} we assume $P$ has class $2$; hence, $p<d$ or $p|q$ and thus the algorithms of \cite[Theorem 3.2]{Luks:mat} can be applied instead of discrete logs.  Thus, there are no discrete log type problems to consider for matrix $p$-groups of class $2$.

\subsection{Abelian $p$-groups, bases, and
solving systems of equations}\label{sec:linear}
We outline the obvious generalizations of linear algebra we require
to work with abelian $p$-groups.  A careful exposition is given in 
\cite[Chapter I,Section I.G]{McDonald:comm-lin}.

Let $V$ be a finite abelian $p$-group.  A set $\mathcal{X}\subseteq V$ is \emph{linearly independent} if $0=\sum_{x\in\mathcal{X}} s_x x$, $s_x\in\mathbb{Z}$, implies $s_x\equiv 0 \mod{|x|}$, for all $x\in\mathcal{X}$.  A \emph{basis} $\mathcal{X}$ for $V$ is a linearly independent generating set of $V$; hence, $V=\bigoplus_{x\in\mathcal{X}} \langle x\rangle$.  Every basis of $V$ determines an isomorphism to an additive representation $\mathbb{Z}_{p^{e_1}}\oplus\cdots\oplus \mathbb{Z}_{p^{e_s}}$ for $e_1\leq \cdots \leq e_s\in\mathbb{Z}^+$.  Operating in the latter representation is preferable to $V$'s original representation and we assume that all abelian groups (including subgroups) are specified with a basis.

Each endomorphism $f$ of $V$ can be represented by an integer matrix 
$F=[F_{ij}]$ such that $p^{e_j-e_i}|F_{ij}$, $1\leq i\leq j\leq s$, and furthermore, every such matrix induces an endomorphism of $V$ 
(with respect to $\mathcal{X}$) \cite[Theorem 3.3]{AutAbel}. 

To row-reduce an $m\times n$ matrix $A$ with entries in $\mathbb{Z}_{p^e}$ is a modification of Gaussian elimination: first sort the rows so that the least residue classes satisfy $A_{11}|A_{i1}$ as integers, for all $i\geq 1$, then continue with standard row reduction noting that it may be impossible to clear entries above a pivot entry.  That 
process uses $O(m^2 n)$ operations in $\mathbb{Z}_{p^e}$ and leads
to algorithms which convert generators of $V$ into a basis, extend
linearly independent subsets of $V$ to bases, and compute the intersection of subgroups.  Improvements on these methods can be had, consider and \cite[Theorem 8.3]{McKenzie-Cook}.  

Recently, P. A. Brooksbank and E. M. Luks created a polynomial-time
algorithm which, given a module $M$ and nontrivial submodule $N$,
returns a direct decomposition $M=X\oplus Y$ with $N\leq X$ and
$X$ minimal with that property \cite[Theorem 3.6]{Brooksbank-Luks:mod}.  We use that result in the specific context of $\mathbb{Z}_{p^e}$-modules.

\subsection{Bilinear maps, $\perp$-decompositions, and isometry}
\label{sec:bi}
A \emph{$\mathbb{Z}_{p^e}$-bilinear map} 
$b:V\times V\to W$ is a function of $\mathbb{Z}_{p^e}$-modules 
$V$ and $W$ where
\begin{equation}
	b(su+u',tv+v') = stb(u,v)+sb(u,v')+tb(u',v)+b(u',v'),
\end{equation}
for each $u,u',v,v'\in V$ and $s,t\in \mathbb{Z}_{p^e}$.  A 
\emph{$\perp$-decomposition} of $b$ is a decomposition $\mathcal{V}$
of $V$ into a direct sum of submodules which are pairwise orthogonal, i.e. $b(X,Y)=0$ for distinct $X,Y\in\mathcal{V}$.

Let $\mathcal{X}$ and $\mathcal{Z}$ be ordered bases of the $V$ 
and $W$ respectively.  Set $B_{xy}^{(z)}\in\mathbb{Z}_{p^e}$ so
\begin{equation}
b\left(\sum_{x\in\mathcal{X}} s_x x,
\sum_{y\in\mathcal{X}} t_y y\right)
=\sum_{x,y\in\mathcal{X}}\sum_{z\in\mathcal{Z}} s_x t_y B_{xy}^{(z)} z,
\quad \forall s_x,s_y\in \mathbb{Z}_{p^e}, x,y\in\mathcal{X}.
\end{equation}
Set
\begin{equation*}
	B_{xy}=\sum_{z\in\mathcal{Z}} B_{xy}^{(z)}z,\qquad 
		\forall x,y\in\mathcal{X};
\end{equation*}
so that $B=[B_{xy}]_{x,y\in\mathcal{X}}$ is an $n\times n$-matrix
with entries in $W$, where $n=|\mathcal{X}|$.  Writing the elements
of $V$ as row vectors with entries in $\mathbb{Z}_{p^e}$ with respect
to the basis $\mathcal{X}$ we can then write:
\begin{equation}\label{eq:bilinear-comp}
	b(u,v)=uBv^t,\qquad \forall u,v\in V.
\end{equation}

Take $F,G\in \End V$ represented as matrices. Define $FB$ and $BG^t$ by the usual matrix multiplication, but notice the result is a matrix with entries in $W$.  Evidently, $(F+G)B=FB+GB$, $F(GB)=(FG)B$, and similarly for the action on the right.  The significance of these operations is seen by their relation to $b$:
\begin{equation}
	b(uf,v) = uFBv^t\textnormal{ and }
	b(u,vg) = uBG^t v^t;
\end{equation}
for all $u,v\in V$.

An \emph{isometry} between two bilinear maps $b:V\times V\to W$ and 
$b':V'\times V'\to W$ is an isomorphism $\alpha:V\to V'$ such that
$b'(u\alpha,v\alpha)=b(u,v)$ for all $u,v\in V$.  Evidently, isometries map $\perp$-decompositions of $b$ to $\perp$-decomposition of $b'$.

Finally, we call a bilinear map \emph{$\theta$-symmetric} if there is
$\theta\in \GL(W)$ of order at most $2$ such that
\begin{equation}
	b(u,v)=b(v,u)\theta,\qquad \forall u,v\in V.
\end{equation}
This meaning of $\theta$-symmetric includes the usual \emph{symmetric}, $b(u,v)=b(v,u)$; and \emph{skew symmetric}, $b(u,v)=-b(v,u)$ flavors of bilinear maps.  If $W=\langle b(u,v) : u,v\in V\rangle$ then $\theta$ is uniquely determined by $b$ and so we make no effort to specify $\theta$ explicitly.

\subsection{Rings}

All our rings are subrings of $\End V$, for a given abelian $p$-group
$V$ (as in Section \ref{sec:linear}).  These rings will be specified by a set of matrices which generate the ring under addition and
multiplication.  Multiplication and addition are handled in the usual matrix manner.

%
%
\section{Reducing central decompositions to orthogonal decompositions}
\label{sec:groups}

In this section we reduce the problem of finding a central decomposition of a $p$-group of class $2$ to the related problem of finding a $\perp$-decomposition of an associated bilinear map.
Throughout we assume that $P$ is a $p$-group of class $2$.

\subsection{The bilinear maps $\Bi(P)$}\label{sec:bi-grp}
R. Baer \cite{Baer:class-2} associated to $P$ various bilinear maps  
including: $b:=\Bi(P)$ defined by $b:P/Z(P)\times P/Z(P)\to P'$ 
where 
\begin{equation}
	b(Z(P)x,Z(P)y):=[x,y],\qquad\forall x,y\in P.
\end{equation}
It is evident that $b$ is $\mathbb{Z}_{p^e}$-bilinear where $p^e=\exp(P)$. Notice that $b$ is alternating: $b(Z(P)x,Z(P)x)=0$, for all $x\in P$.

\begin{remark}
Since $P'$ can be an arbitrary finite abelian $p$-group, the bilinear map $\Bi(P)$ is rarely a bilinear form.  That occurs only if $P$ is an extraspecial or almost extraspecial $p$-group (i.e.: $\mathbb{Z}_p\cong P'=\Phi(P)\leq Z(P)$).  Those examples are important to consider but highly atypical of the general setting.
\end{remark}

\subsection{Central decompositions from orthogonal decompositions}\label{sec:cent-prod}
Let $\mathcal{H}$ be a central decomposition of $P$.   The following 
related sets are useful:
\begin{align}
\mathcal{H}Z(P) & := \{ HZ(P) : H\in\mathcal{H} \} - \{ Z(P)\},
	\label{eq:HM}\\
\mathcal{H}Z(P)/Z(P) & := \{ HZ(P)/Z(P) : H\in\mathcal{H}\} 
	- \{Z(P)/Z(P)\},	\textnormal{ and }\\
Z(\mathcal{H}) & := \{ H\in\mathcal{H} : H\leq Z(P) \}.
\end{align}
Note that $\mathcal{H}-Z(\mathcal{H})$ is in bijection with
$\mathcal{H}Z(P)/Z(P)$ so that
\begin{equation}\label{eq:count}
	|\mathcal{H}|=|\mathcal{H}Z(P)/Z(P)|+|Z(\mathcal{H})|.
\end{equation}
Since $P=\langle \mathcal{H}\rangle$ and $[H,\langle \mathcal{H}-\{H\}]=1$, it follows that $H\intersect \langle \mathcal{H}-\{H\}\rangle\leq Z(P)$.  Thus, $\mathcal{H}Z(P)/Z(P)$ is a direct decomposition of $P/Z(P)$.

Suppose that $\mathcal{V}$ is a direct decomposition of $P/Z(P)$.
Define
\begin{equation}\label{eq:pullback}
\mathcal{H}(\mathcal{V})
:=\{ H\leq P:  Z(P)\leq H, H/Z(P)\in \mathcal{V}\}.
\end{equation}
Note that $\mathcal{V}$ and $\mathcal{H}(\mathcal{V})$ are in a 
natural bijection.

\begin{prop}\label{prop:cent-perp}
Let $P$ be a $p$-group of class $2$ and $b:=\Bi(P)$.
\begin{enumerate}[(i)]
\item If $\mathcal{H}$ is a central decomposition of $P$ then 
$\mathcal{H}Z(P)/Z(P)$ is a $\perp$-decomposition of $b$.
\item If $\mathcal{V}$ is a $\perp$-decomposition of $b$ then 
$\mathcal{H}(\mathcal{V})$ is a central decomposition of $P$ where 
$\mathcal{H}(\mathcal{V})Z(P)=\mathcal{H}(\mathcal{V})$ and 
$\mathcal{H}(\mathcal{V})/Z(P)=\mathcal{V}$.
\end{enumerate}
\end{prop}
\begin{proof}
$(i)$.
If $\mathcal{H}$ is a central decomposition of $P$ then
$\mathcal{H}Z(P)/Z(P)$ is a direct decomposition of $V:=P/Z(P)$.
Furthermore, if $H$ and $K$ are distinct members of
$\mathcal{H}$ then $[H,K]=1$, so that 
$b(HZ(P)/Z(P),KZ(P)/Z(P))=0$.
Thus, $\mathcal{H}Z(P)/Z(P)$ is a $\perp$-decomposition of $b$.

$(ii)$.
Let $\mathcal{V}$ be a $\perp$-decomposition of $b$ and
set $\mathcal{K}:=\mathcal{H}(\mathcal{V})$.  By definition,
$\mathcal{K}=\mathcal{K}Z(P)$ and
$\mathcal{K}/Z(P)=\mathcal{V}$, so that
$K\intersect \langle \mathcal{K}-\{K\}\rangle=Z(P)$ for all
$K\in\mathcal{K}$. It remains to show that $\mathcal{K}$ 
is a central decomposition of $P$.
As $\mathcal{V}\neq\emptyset$ it follows that 
$\mathcal{K}\neq \emptyset$.
Furthermore, $V=\langle\mathcal{V}\rangle$ so 
$P=\langle \mathcal{K},Z(P)\rangle
=\langle \mathcal{K}\rangle$,
as $Z(P)\leq K$ for any $K\in\mathcal{K}$.  Since 
$\mathcal{K}$ is in bijection
with $\mathcal{V}$, if $\mathcal{J}$ is a proper
subset of $\mathcal{K}$ then $\mathcal{J}/Z(P)$ is a 
proper subset of $\mathcal{V}$ and as $\mathcal{J}/Z(P)$ does not 
generate $V$ it follows that $\mathcal{J}$ does 
not generate $P$.  Finally, if $H$ and $K$ are distinct members of 
$\mathcal{K}$ then $0=b(H/M,K/M)=[H,K]$.  Thus, 
$\mathcal{K}$ is a central decomposition of $P$.
\end{proof}

\begin{thm}\label{thm:cent-indecomp}
If $P$ is a $p$-group of class $2$, then
$P$ is centrally indecomposable if, and only if, $\Bi(P)$ is 
$\perp$-indecomposable and $Z(P)\leq \Phi(P)$.
\end{thm}
\begin{proof}
Assume that $P$ is centrally indecomposable.

Let $\mathcal{V}$ be a $\perp$-decomposition of $\Bi(P)$.
By \propref{prop:cent-perp}.$(ii)$, $\mathcal{H}(\mathcal{V})$
is a central decomposition of $P$ and therefore 
$\mathcal{H}(\mathcal{V})=\{P\}$.  Hence, 
$\mathcal{V}=\mathcal{H}(\mathcal{V})/Z(P)=\{P/Z(P)\}$.  As $\mathcal{V}$
was an arbitrary $\perp$-decomposition of $\Bi(P)$, it follows that
$\Bi(P)$ is $\perp$-indecomposable. 

Next let $\Phi(P)\leq Q\leq P$ be such that 
$P/\Phi(P)=Q/\Phi(P)\oplus Z(P)\Phi(P)/\Phi(P)$ as $\mathbb{Z}_p$-vector
spaces.  Set $\mathcal{H}=\{Q,Z(P)\}$.  Clearly $[Q,Z(P)]=1$ and
$P$ is generated by $\mathcal{H}$.  Therefore, $\mathcal{H}$ contains
a subset which is a central decomposition of $P$.  As $P$ is centrally
indecomposable and $P\neq Z(P)$, it follows that $P=Q$, and so
$1=Z(P)\Phi(P)/\Phi(P)$, so
that $Z(P)\leq \Phi(P)$.

For the reverse direction we assume that $\Bi(P)$ is $\perp$-indecomposable and that $Z(P)\leq \Phi(P)$.
Let $\mathcal{H}$ be a central decomposition of $P$.  

By \propref{prop:cent-perp}.$(i)$ we know $\mathcal{H}Z(P)/Z(P)$ is a 
$\perp$-decomposition of $\Bi(P)$.  Thus, 
$\mathcal{H}Z(P)/Z(P)=\{P/Z(P)\}$
so that $\mathcal{H}Z(P)=\{P\}$.  Hence, for all $H\in\mathcal{H}$,
either $H\leq Z(P)$ or $HZ(P)=P$.  As $Z(P)\leq \Phi(P)<P$, it 
follows that at
least one $H\in \mathcal{H}$ is not contained in $Z(P)$ and furthermore,
$P=HZ(P)=H$ as $Z(P)$ consists of non-generators.  Since no proper subset of
$\mathcal{H}$ generates $P$ and $P\in\mathcal{H}$, it follows that
$\mathcal{H}=\{P\}$.  Since $\mathcal{H}$ was an arbitrary central
decomposition of $P$ it follows that $P$ is centrally indecomposable.
\end{proof}

\begin{lemma}
For a $p$-group $P$ of class $2$ where $\Bi(P)$ is
$\perp$-indecomposable, every central decomposition of $P$ 
has exactly one nonabelian member.
\end{lemma}
\begin{proof}
Let $\mathcal{H}$ be  central decomposition of $P$.
Since $P\neq Z(P)$ and $\Bi(P)$ is $\perp$-indecomposable,
there is a nonabelian $H\in\mathcal{H}$ and $\mathcal{H}Z(P)=\{P\}$
so that $P=HZ(P)$.  If $K\in\mathcal{H}-\{H\}$ then $[K,P]=[K,HZ(P)]=[K,H]=1$, since distinct members of $\mathcal{H}$ commute.  Thus $K\leq Z(P)$, which proves that $H$ is the only nonabelian group in $\mathcal{H}$.
\end{proof}

\begin{prop}\label{prop:reduce}
There is a deterministic polynomial-time algorithm which,
given a $p$-group $P$ of class $2$ such that $\Bi(P)$ is
$\perp$-indecomposable, returns a nonabelian centrally 
indecomposable group $Q$ such that $P=Q$ or $\{Q,Z(P)\}$ 
is a central decomposition of $P$.
\end{prop}
\begin{proof}
\emph{Algorithm.}
If $Z(P)\leq \Phi(P)$ then return $P$; otherwise, compute generators 
for a vector space complement $Q/\Phi(P)$ to $Z(P)\Phi(P)/\Phi(P)$ in
$P/\Phi(P)$, $\Phi(P)\leq Q<P$.   Recurse with $Q$ in the 
r\^{o}le of $P$ and return the result of this recursive call.

\emph{Correctness.}
If $Z(P)\leq \Phi(P)$ then \thmref{thm:cent-indecomp} proves 
that $P$ is centrally indecomposable.  Otherwise, $Z(P)\Phi(P)/\Phi(P)$
is a proper subspace of the vector space $P/\Phi(P)$.  The group $Q$ satisfies
$P=QZ(P)$.  Hence, $P'=[Q Z(P),Q Z(P)]=Q'$ (so $Q$ is nonabelian)
and $[Z(Q),P]=[Z(Q),Q Z(P)]=1$, so that $Z(Q)=Q\intersect Z(P)\geq P'$.  
In particular, the isomorphism of $P/Z(P)=QZ(P)/Z(P)\cong 
Q/Z(P)\intersect Q=Q/Z(Q)$ gives an isometry between $\Bi(P)$ and 
$\Bi(Q)$ which implies that $\Bi(Q)$ is $\perp$-indecomposable. 
Thus we may recurse with $Q$.  By induction, the return of a recursive
call is a centrally indecomposable subgroup $P'\leq R\leq P$ such that
$Q=RZ(Q)$ and so $P=RZ(P)$, which proves that $\{R,Z(P)\}$ is
a central decomposition of $P$.

\emph{Timing.}
The number of recursive calls is bounded by the log of the exponent 
$p^e$ of $P/P'$.  To find a vector space complement amounts to finding a basis of $Z(P)\Phi(P)/\Phi(P)$ and extending the basis to one for 
$P/\Phi(P)$, and so it uses $O(\log^3 [P:\Phi(P)])$ operations in 
$\mathbb{Z}_p$.  Hence, the total number of operations in
$\mathbb{Z}_p$ is in $O(e \log^3 [P:\Phi(P)])
\subseteq O(\log^4 [P:P'])$.
\end{proof}

\begin{coro}\label{coro:convert}
There are deterministic polynomial-time algorithms which, 
given a $p$-group $P$ of class $2$ and $\mathcal{V}$ a fully 
refined $\perp$-decomposition of $\Bi(P)$, return a
fully refined central decomposition $\mathcal{J}$ of $P$ such that:
\begin{enumerate}[(i)]
\item $\mathcal{J}Z(P)/Z(P)=\mathcal{V}$ and
\item $Z(\mathcal{J})$ is a direct decomposition of $Z(P)$.
\end{enumerate}
In particular, if $\mathcal{V}$ has maximum size amongst the set
of $\perp$-decompositions of $\Bi(P)$, then $\mathcal{H}$ has
maximum size amongst the set of central decompositions of $P$.
\end{coro}
\begin{proof}
\emph{Algorithm.}
Compute the pullback $\mathcal{H}:=\mathcal{H}(\mathcal{V})$.  
Set $\mathcal{K}=\emptyset$.  For each $H\in\mathcal{H}$, 
use the algorithm of \propref{prop:reduce} to find a nonabelian 
centrally indecomposable subgroup $K\leq H$ such that $H=KZ(P)$ 
and add $K$ to $\mathcal{K}$.  Next, find bases for $Z(P)$ and
for $Z(\langle \mathcal{K}\rangle)$ and apply the algorithm for
\cite[Theorem 3.6]{Brooksbank-Luks:mod} to find a direct 
factor $X$ of $Z(P)$ which is minimal with respect to containing
$Z(\langle\mathcal{K}\rangle)$.  Find a basis $\mathcal{X}$ for $X$
and $\mathcal{Y}$ of a complement $Y$ to $X$ in $Z(P)$, and return 
\begin{equation}
\mathcal{J}:=\mathcal{K}\sqcup \{\langle x\rangle : x\in\mathcal{X},
x\notin Z(\langle \mathcal{K}\rangle)\}\sqcup
\{\langle y\rangle:y\in\mathcal{Y}\}.
\end{equation}

\emph{Correctness.}  By \propref{prop:cent-perp}
we know that $\mathcal{H}$ is a central decomposition of $P$ in which
every member $H$ has $Z(H)=Z(P)$ and $\Bi(H)$ is $\perp$-indecomposable.  Thus the algorithm of \propref{prop:reduce} 
can be applied to $H$ and the set $\mathcal{K}$ consists of nonabelian 
centrally indecomposable subgroups where distinct members pairwise commute; thus, $\mathcal{K}$ is a fully refined central decomposition of $\langle \mathcal{K}\rangle$ of maximum possible size.  Notice $\mathcal{K}=\mathcal{J}-Z(\mathcal{J})$ and the members of $Z(\mathcal{J})$ are cyclic and a direct decomposition of $\langle Z(\mathcal{J})\rangle$. Hence, $\mathcal{J}$ is a fully refined central decomposition of $P$.  Furthermore, $\mathcal{K}Z(P)=\mathcal{H}$.  By \propref{prop:cent-perp}.$(ii)$ we have:
\begin{eqnarray}
	\mathcal{J}Z(P)/Z(P) & = & \mathcal{K}Z(P)/Z(P)
		=\mathcal{H}/Z(P)=\mathcal{V}\textnormal{ and }\\
	Z(\mathcal{J}) & = & \{\langle x\rangle : x\in\mathcal{X},
x\notin \langle \mathcal{K}\rangle\}\sqcup
\{\langle y\rangle:y\in\mathcal{Y}\}.
\end{eqnarray}
Thus, $(i)$ and $(ii)$ is proved.  It remains to prove that $\mathcal{J}$ has maximum size amongst central decompositions of $P$.

First $|\mathcal{J}|=|\mathcal{K}|+|Z(\mathcal{J})|$.  Also, $X$ is a minimal direct factor of $Z(P)$ which contains $P'$ and so $Z(P)^p P'=Z(P)^p X$.  As, $Z(P)=X\oplus Y$ and $\langle \mathcal{K}\rangle'=P'$, it follows that $|Z(\mathcal{J})|=\rank Z(P)-\rank Z(P)^p P'/Z(P)^p$.  If $\mathcal{L}$ is 
any other central decomposition of $P$, then $|\mathcal{L}|=|\mathcal{L}Z(P)/Z(P)|+|Z(\mathcal{L})|$. By the maximality of $\mathcal{V}$, 
$|\mathcal{L}Z(P)/Z(P)|\leq |\mathcal{V}|=|\mathcal{K}|$.  As
$Z(P)$ is abelian and $\langle Z(\mathcal{L})\rangle\leq Z(P)$,
it follows that $|Z(\mathcal{L})|\leq \rank Z(P)-\rank Z(P)^p P'/Z(P)^p$.  Thus, $\mathcal{J}$ has maximum possible size.

\emph{Timing.}
There are $|\mathcal{V}|$ calls made to the
algorithm of \propref{prop:reduce}, which uses 
$O(\log \exp(H)\log^3 [H:\Phi(H)])$  operations in $\mathbb{Z}_p$ for each $H\in \mathcal{H}$.  The algorithm of \cite[Theorem 3.6]{Brooksbank-Luks:mod} runs in $O(\rank^6 Z(P))$-time.  Thus, the number of field operations lies in $O(|\mathcal{V}|\log \exp(P)\log^3 [P:\Phi(P)]
+|S|\memb(P)+\rank^6 Z(P))\subseteq O(\log^5 [P:P']+\log^6 Z(P)/Z(P)^p)$.
\end{proof}

%
%
\section{The $*$-ring of adjoints of a bilinear map}\label{sec:adj}
The translations of Section \ref{sec:groups} lead us to consider
how to find a fully refined $\perp$-decomposition of a bilinear map.
For this we introduce the ring of adjoints.

Throughout this section we assume that $b:V\times V\to W$ is a 
$\theta$-symmetric $\mathbb{Z}_{p^e}$-bilinear map.

\subsection{Adjoints: $\Adj(b)$, $\Sym(B)$, and $\mathfrak{H}(R,*)$}
The ring of \emph{adjoints} of $b$ is:
\begin{equation}\label{eq:adj}
	\Adj(b)  := \{(f,g)\in\End V \oplus (\End V)^{op}:
		 b(uf,v)=b(u,vg), \forall u, v\in V\}.
\end{equation}
There is a natural subset of $\Adj(b)$ of \emph{self-adjoint} elements:
\begin{equation}
	\Sym(b) := \{(f,f)\in\End V \oplus (\End V)^{op}:
		 b(uf,v)=b(u,vf),  \forall u,v\in V\}.
\end{equation}

\begin{remark}
Notice that $\Sym(b)$ is not an associative subring but rather
a Jordan ring, quadratic in the case of characteristic $2$, cf. 
\cite[Section 4.5]{Wilson:unique}.  This is a vital observation for
answering questions surrounding $\perp$-decompositions; however, for
algorithmic purposes this nonassociative perspective is not necessary.
\end{remark}

If $b$ is $\theta$-symmetric then $(f,g)\in\Adj(b)$ if, and only if, 
$(g,f)\in \Adj(b)$.  Hence, $(f,g)\mapsto (g,f)$ is an anti-isomorphism $*$. Indeed, $*$ has order $1$ or $2$ so that $\Adj(b)$ is a $*$-ring.

In general, for a $*$-ring $(R,*)$ and additive subgroup $S\subseteq R$, we define $\mathfrak{H}(S,*)=\{s\in S: s^*=s\}$ which is again a subgroup of $S$ as $*$ is additive.  ($\mathfrak{H}$ is for Hermite
and is a notation encouraged by Jacobson \cite[Section 1.4]{Jacobson:Jordan}.)  Evidently, $\Sym(b)=\mathfrak{H}(\Adj(b))$.

\subsection{Self-adjoint idempotents}

An endomorphism $e\in \End V$ is an \emph{idempotent}
if $e^2=e$.  This makes $V=Ve\oplus V(1-e)$.
Indeed, every direct decomposition $\mathcal{V}$ of $V$ is parameterized by the set $\mathcal{E}:=\mathcal{E}(\mathcal{V})$ of projection idempotents; that is, for each $U\in\mathcal{V}$, $e_U\in \mathcal{E}$ with kernel $\langle \mathcal{V}-\{U\}\rangle$ and where the restriction of $e_U$ to $U$ is the identity. It follows that distinct members $e$ and $f$ of $\mathcal{E}$ are \emph{orthogonal} (i.e. $ef=0=fe$) and $1=\sum_{e\in\mathcal{E}} e$.

Evidently $1\in \Sym(b)$, so $\Sym(b)$ contains idempotents.  
All idempotents in $\Sym(b)$ are self-adjoint.  The significance of 
$\Sym(b)$ is the following:
\begin{thm}\label{thm:idemp}
A direct decomposition $\mathcal{V}$ of $V$ is a $\perp$-decomposition of $b:V\times V\to W$ if, and only if, $\mathcal{E}(\mathcal{V})\subseteq \Sym(b)$.
\end{thm}
\begin{proof}
Suppose that $\mathcal{V}$ is a $\perp$-decomposition of $b$.
Take $e\in\mathcal{E}(\mathcal{V})$.  Then 
$b(ue,v)=b(ue,ve+v(1-e))=b(ue,ve)+b(ue,v(1-e))$,
for all $u,v\in V$.  As $1-e=\sum_{f\in\mathcal{E}(\mathcal{V})-\{e\}}f$, and $Ve$ is perpendicular to $Vf$ for each $f\in\mathcal{E}(
\mathcal{V})$, it follows that $Ve$ is perpendicular to $V(1-e)$;
hence, $b(ue,v)=b(ue,ve)$.  Similarly,
$b(ue,ve)=b(u,ve)$, so that $e\in \Sym(b)$.

Now suppose that $\mathcal{V}$ is a direct decomposition of $V$
with $\mathcal{E}(\mathcal{V})\subseteq\Sym(b)$.  If $e\in\mathcal{E}(\mathcal{V})$ then
$b(ue,v(1-e))=b(u,v(1-e)e)=0$, for all $u,v\in V$.  
So $Ve$ is perpendicular to
$V(1-e)$.  Thus every subspace of $V(1-e)$ is perpendicular to
$Ve$, which includes $Vf$ for every $f\in\mathcal{E}(\mathcal{V})-\{e\}$.  So $\mathcal{V}$ is a $\perp$-decomposition of $b$.
\end{proof}

A self-adjoint idempotent $e\in \Sym(b)$ is \emph{self-adjoint-primitive} if it is not the sum of proper (i.e.: not $0$ or $1$) pairwise orthogonal self-adjoint idempotents in $\Sym(b)$.  (Such idempotents need not be primitive in $\Adj(b)$ under the usual
meaning of primitive idempotents.)
A set of pairwise orthogonal self-adjoint-primitive idempotents 
of $\Sym(b)$ which sum to $1$ is called a \emph{frame} of $\Sym(b)$.  More generally, in a $*$-ring
$(R,*)$, a self-adjoint frame is a set of self-adjoint-primitive
pairwise orthogonal idempotents which sum to $1$.

\begin{coro}\label{coro:perp-idemp}
There is a natural bijection between the set of fully refined 
$\perp$-decompositions of $b$ and the set of all frames of $\Sym(b)$.
\end{coro}
\begin{proof}
This follows directly from \thmref{thm:idemp}.
\end{proof}

		%
		%
\subsection{Computing $\Adj(b)$ and $\Sym(b)$}\label{sec:adj-sym-algo}
Let $V$ and $W$ be finite abelian $p$-groups specified with bases 
$\mathcal{X}$ and $\mathcal{Z}$ respectively.  Take $b:V\times V\to W$ 
to be a $\mathbb{Z}_{p^e}$-bilinear map.  Assume
that $b$ is input with structure constant matrix $B$ with respect to the 
bases $\mathcal{X}$ and $\mathcal{Z}$ (cf. \eqref{eq:bilinear-comp}).

If $\End V$ is expressed as matrices (see Section \ref{sec:linear}) with
respect to $\mathcal{X}$ then
\begin{equation}
	\Adj(B) = 
		\{(X,Y)\in \End V\oplus \End V
		 : XB=BY^t\}.
\end{equation}
To find a basis for $\Adj(B)$ we solve for $X$ and $Y$ such that:
\begin{equation}\label{eq:adj-comp}
0 = \sum_{x\in \mathcal{X}} X_{xx'} B_{x' y}^{(z)} 
  	-  \sum_{y\in\mathcal{X}}Y_{yy'} B_{xy'}^{(z)},
	\qquad \forall x,y\in\mathcal{X}, z\in\mathcal{Z}.
\end{equation}
This amounts to solving $|\mathcal{X}|^2 |\mathcal{Z}|$ linear 
equations over $\mathbb{Z}_{p^e}$, each in $2|\mathcal{X}|$ variables
and can be done using $O(|\mathcal{X}|^4 |\mathcal{Z}|)$ operations
in $\mathbb{Z}_{p^e}$ (cf. Section \ref{sec:linear}).
Computing a basis of $\Sym(b)$ can be done in similar fashion.

\begin{remark}
If $b$ is $\theta$-symmetric then the number of equations 
determining $\Adj(b)$
can be decreased by $2$ by considering the ordering of the basis $\mathcal{X}$ and
using only the equations \eqref{eq:adj-comp} for $x\leq y$, $x,y\in\mathcal{X}$ 
and $z\in\mathcal{Z}$.
\end{remark}

%
%
\section{Algorithms for $*$-rings}\label{sec:*-rings}

In Section \ref{sec:adj}, the self-adjoint idempotents of the
$*$-ring $\Adj(b)$ where linked with $\perp$-decompositions
of $b$, and through the theorems of Section \ref{sec:groups},
also to central decompositions of $P$.  
In this section we show how to find self-adjoint idempotents 
by appealing to the semisimple and
radical structure of $*$-rings.  Most of the algorithms
reduce to known algorithms for the semisimple and radical structure theorems of finite algebras over $\mathbb{Z}_p$.  

\subsection{A fast Skolem-Noether algorithm}
Let $K$ be a field of characteristic $p$.  The Skolem-Noether theorem 
states that every ring automorphism $\varphi$ of $M_n(K)$ satisfies 
$X\varphi=D^{-1} X^\sigma D$ for some $(D,\sigma)\in 
\GL_n(K)\rtimes \Gal(K/\mathbb{Z}_p)$, and for all $X\in M_n(K)$,
\cite[(3.62)]{Curtis-Reiner}.  
Given an effective automorphism $\varphi$, there is a straightforward method
to find $(D,\sigma)$ which involves solving a system of $n^2$ linear equations
over $K$ and thus uses $O(n^6)$ field operations. We offer the following
improvement by analyzing the proof the the Skolem-Noether theorem
in \cite[Chapter VIII]{Jacobson:linear}.
\begin{prop}\label{prop:Skolem-Noether}
Given an effective ring automorphism $\varphi$ of $M_n(K)$, $K$ a finite field
of characteristic $p$, there is a deterministic algorithm using 
$O(n^4+\dim_{\mathbb{Z}_p} K)$ operations in $\mathbb{Z}_p$ 
which returns 
$(D,\sigma)\in\GL_n(K)\rtimes\Gal(K/\mathbb{Z}_p)$
such that $X\varphi= D^{-1} X^\sigma D$, for all $X\in M_n(K)$.
\end{prop}
\begin{proof}
\emph{Algorithm.}
Define $g:K^n\to M_n(K)$ by $x\mapsto \begin{bmatrix} x\\ 0\\ \vdots\end{bmatrix}$ and 
$\tau:K^n\to M_n(K)$ by $x\tau=xg\varphi$.  Fix a basis $\{x_1,\dots,x_n\}$
of $K^n$ and find the first $1\leq i\leq n$ such that $x_i (x_j\tau)\neq 0$
for all $1\leq j\leq n$.  Set $D:=\begin{bmatrix} 
x_i(x_1\tau)\\ \vdots \\ x_i(x_n\tau)\end{bmatrix}
\in M_n(K)$.  Induce $\sigma:K\to K$ by $\alpha\mapsto 
[(\alpha I_n)\varphi]_{11}$, then return $(D,\sigma)$.

\emph{Correctness.}
We summarize how the steps in this algorithm perform the various stages 
of the proof of Skolem-Noether, given in \cite[Chapter VIII]{Jacobson:linear}.

Let $I$ be the image of $g$.  As $I$ is a minimal right ideal, the image $J:=I\varphi$ is also a minimal right ideal.  Thus, 
there is an $1\leq i\leq n$ such that $x_i J\neq 0$.  Since
$x_i J$ is a simple right $M_n(K)$-module, it follows that 
$x_i J\cong K^n$.  As $\{x_1g,\dots,x_ng\}$ is a basis of $I$, 
$\{x_1\tau,\dots,x_n\tau\}$ is a basis of $J$ and so
$\{x_i(x_1\tau),\dots,x_i(x_n\tau)\}$ is a basis of $x_i J$.  Thus 
$D$ is an invertible matrix in $M_n(K)$.  Finally, 
$(\alpha I_n)\varphi=(\alpha\sigma)I_n$, for $\alpha\in K$, defines
a field automorphisms of $K$.  It follows that $X\varphi=D^{-1} X^\sigma D$ 
for each $X\in M_n(K)$.

\emph{Timing.}
The algorithm searches over the set of all $1\leq i,j\leq n$ and tests whether
$x_i(x_j\tau)\neq 0$, a test which uses $O(n^2)$ field operations in $K$.  
The additional task of inducing $\sigma$ uses $O(\dim_{\mathbb{Z}_p} K)$ 
operations in $\mathbb{Z}_p$.
\end{proof}

\subsection{Constructive recognition of finite simple $*$-rings}
\label{sec:*-simple}
The classification (up to $*$-isomorphism) of simple 
$*$-rings appears to have developed from multiple disciplines 
simultaneously (most involving rings over infinite or
arbitrary fields).  Key players included A.A. Albert, N. Jacobson,
and A. Weil; see \cite{Lewis:*-survey}.  We attempt to 
give an ersatz proof which condenses the various ideas distributed
amongst the sources.  In particular, we include the elements
that will be used in our algorithms.

\begin{lemma}\label{lem:J(R)}
The Jacobson radical of a $*$-ring is a $*$-ideal.
\end{lemma}
\begin{proof} The Jacobson radical is the intersection over the
set of maximal left ideals as well as the set of maximal right 
ideals;  $*$ interchanges these sets.
\end{proof}

\begin{thm}\label{thm:*-coord}
A finite $*$-simple ring $(R,*)$ is either simple as a ring or
the direct product of two isomorphic simple rings.  Thus, 
there is a field $K$, a vector
space $V$ over $K$, and an involution 
$\circ$ on $\End_{K}V$ such that $(R,*)$
is $*$-ring isomorphic to one of the following:
\begin{enumerate}[(i)]
\item \emph{Classical}: $(\End_{K} V,\circ)$, 
\item \emph{Exchange}: 
$(\End_{K} V\oplus \End_{K} V,\bullet:=\circ\wr 2)$ where
$(f,g)^{\bullet}=(g^{\circ},f^{\circ})$, for all $f,g\in \End_K V$.
Furthermore, any two exchange type $*$-simple $*$-rings which
are isomorphic as rings are isomorphic as $*$-rings.
\end{enumerate}
\end{thm}
\begin{proof}(The proof is implicit in \cite[p. 178]{Jacobson:Jordan}.)
By Lemma \ref{lem:J(R)}, $J(R)$ is a $*$-ideal.  As $(R,*)$ is 
$*$-simple, $J(R)=0$.  By the Wedderburn theorems, $R$ is a 
direct product of its minimal ideals.  
Fix a minimal ideal $M$ of $R$ and $I$ a minimal left ideal
of $M$.  Thus, $M^*$ is also a minimal ideal of $R$ with 
minimal right ideal $I^*$.  As $M$ is a simple ring its center $K:=Z(M)$
is a field.  Evidently $I$ is a left $K$-vector space and
by Wedderburn's theorems, the left action of $M$ on $I$ produces 
a ring isomorphism $\varphi:M\to \End_K I$.  Define, 
$\varrho:M^*\to \End_K I$ by $v(x\varrho):=x^*v$ for all $x\in M^*$ 
and $v\in I$.  Evidently $\varrho$ is also a ring isomorphism. 
Thus, $f\mapsto f^\circ:=(f\varphi^{-1})^*\varrho$,
for all $f\in\End_K I$, is an involution on $\End_K I$.

Finally, $M+M^*$ is a nontrivial $*$-ideal and $(R,*)$ is $*$-simple;
therefore, $R=M+M^*$.  If $M=M^*$ then $(R,*)$ is of classical
type and $\varphi$ is a $*$-ring isomorphism to $(\End_K I,\circ)$.  Otherwise, $R=M\oplus M^*$ and $(R,*)$ is of exchange type and 
$\varphi\oplus \varrho$ is a $*$-ring isomorphism to 
$(\End_K I\oplus \End_K I, \bullet)$.

If $*$ is another involution on $\End_K I$ and $\diamond:=*\wr 2$.
Define $\mu:\End_K I\oplus \End_K I\to \End_K I \oplus \End_K I$ by
\begin{equation}\label{eq:exchange-iso}
(f,g)\mapsto (f,g^{\circ *}),\qquad \forall f,g\in\End_K I.
\end{equation}
Evidently $\mu$ is a ring isomorphism.  Furthermore,
\begin{equation}
(f,g)^{\bullet}\mu 
	= (g^{\circ},f^{\circ \circ *})
	= (f,g^{\circ *})^{\diamond}
	= (f,g)\mu^{\diamond},\qquad \forall f,g\in\End_K I.
\end{equation}
Thus $\mu$ is a $*$-ring isomorphism.
\end{proof}
\begin{remark}
The map $\mu$ defined in \eqref{eq:exchange-iso} need not be
$K$-linear, but rather only $K$-semilinear.  Our algorithms do 
not require $K$-linear isomorphisms, but they can be modified
to detect these distinctions when necessary.
\end{remark}

From the coordinatization of $*$-simple algebras given in 
\thmref{thm:*-coord}, it is now an application of the Skolem-Noether
theorem and classical forms to produce the following 
\propref{prop:intoAdj}; compare \cite[IX.10-11]{Jacobson:linear}.
\begin{prop}\label{prop:intoAdj}
There is a deterministic polynomial-time algorithm which, 
given a finite classical $*$-simple $*$-ring
$(M_n(K),\circ)$, returns a $*$-ring isomorphism 
$\varphi:(M_n(K),\circ)\to \Adj(d)$, where $d:K^n\times K^n\to K$ is
a nondegenerate symmetric or alternating bilinear, or Hermitian sesquilinear $K$-form.
\end{prop}
\begin{proof}
\emph{Algorithm.}
Apply the algorithm of \propref{prop:Skolem-Noether} to 
$\circ$ to find $(D,\sigma)\in \GL_n(K)\rtimes \Gal(K/\mathbb{Z}_p)$ 
such that $(X^{\circ})^t=D^{-1} X^\sigma D$, for all $X\in M_n(K)$.
If $\sigma\neq 1$ and $D^t D^{-\sigma}=-I$ then find $\gamma\in K$ such
that $\gamma^{\sigma}\neq \gamma$, and reset $D:=(\gamma-\gamma^{\sigma})D$.
Define $d:K^n\times K^n\to K$ by $d(u,v):=uD(v^{\sigma})^t$,
for all $u,v\in K^n$.  Return  $\mu:(M_n(K),\circ)\to \Adj(d)
\subseteq \End_K K^n\oplus \End_K K^n$ 
defined by $a\mu:=(a,a^{\circ})$.

\emph{Correctness.} As $D$ is invertible, $d$ is biadditive, linear
in the first variable, and nondegenerate.
For all $\alpha\in K$,
$\alpha^{\sigma^2} I=D^{-2}(\alpha^{\sigma^2} I)D^2
		=((\alpha I)^{\circ})^{\circ}=\alpha I$.
Hence, $\sigma^2=1$.  Also,
\begin{equation}
	X=(X^{\circ})^{\circ} = D^t(D^t X^{\sigma t} D^{-t})^{\sigma t} D^{-t}
		= D^t D^{-\sigma} X D^{\sigma} D^{-t},\qquad \forall X\in M_n(K).
\end{equation}
Thus, $D^{-\sigma} D^{t}=\alpha I$, for some $\alpha\in K$.  As $D=(D^t)^t=\alpha^2 D$, it follows that $\alpha=\pm 1$.  Therefore $D^t=\pm D^{\sigma}$.  If $\sigma=1$ then $d$ is $\pm 1$-symmetric. 
If $\sigma\neq 1$ and $\alpha =1$ then $d$ is Hermitian.  Otherwise,
$\alpha=-1\neq 1$, $\chr K\neq 2$, and $K$ is a 
quadratic field extension over the subfield fixed by $\sigma$.  So 
there is a $\beta:=\gamma-\gamma^\sigma\in K$ such that $\beta^{\sigma}=-\beta$.  Evidently,
$(\beta D)^t (\beta D)^{-\sigma}=-\beta \beta^{-\sigma} I=I$.  Thus,
reseting $D$ to $\beta D$ makes $D^t=D^{\sigma t}$ and $d$ is
Hermitian.

Finally, $XD=D(X^{\circ})^{\sigma t}$ so that $d(uX,v)=d(u,vX^{\circ})$
for each $X\in M_n(K)$ and $u,v\in K^n$.  Thus $(M_n(K),\circ)$ is 
$*$-isomorphic to $\Adj(d)$ via $X\mapsto (X,X^{\circ})$.

\emph{Timing.} Applying the algorithm for \propref{prop:Skolem-Noether}
uses $O(n^4+\dim_{\mathbb{Z}_p} K)$ operations
in $\mathbb{Z}_p$.  Determining if $\sigma\neq 1$ discovers some
$\gamma\in K$ such that $\gamma^{\sigma}\neq \gamma$, and can
be carried out within the algorithm of \propref{prop:Skolem-Noether}.
Therefore, the remaining computations involve only matrix multiplication.
So the overall time lies in $O(n^4 +\dim_{\mathbb{Z}_p} K)$.
\end{proof}

\begin{remark}
The $*$-simple $*$-rings of exchange type can also be treated
as adjoints of a form.  Specifically, let 
$C:=(K\oplus K,\overline{(\alpha,\beta)}:=(\beta,\alpha))$.
Then define $d:C^n\times C^n\to C$ by $d(u,v)=u\bar{v}^t$.
Evidently, $\Adj(d)$ is $*$-ring isomorphic to 
$(M_n(K)\oplus M_n(K),(X,Y)^\bullet:=(Y^t, X^t))$.
See \cite[Section 4.2]{Wilson:unique} for uniform treatment 
of these forms using associative composition algebras.
\end{remark}

\subsection{Computing the $*$-semisimple and $*$-radical structure
of $\Adj(b)$}

We require the following generalization of the algorithm of 
\cite{Ivanyos:fast-alge} using effective homomorphism (Section \ref{sec:linear}).
\begin{thm}\label{thm:simple-factors}
There is a Las Vegas polynomial-time algorithm which, 
given $R\subseteq \End V$, for a finite abelian $p$-group $V$,
returns a set $\Omega$ of effective ring
epimorphisms such that:
\begin{enumerate}[(i)]
\item for each $\pi\in \Omega$, $\ker \pi$ is a maximal ideal of $R$ 
and the image of $\pi$ is $M_n(K)$ for some field $K$ and integer
$n$ (dependent on $\pi$),
\item for each maximal ideal $M$ of $R$ there is a unique $\pi\in\Omega$
such that $M=\ker \pi$, and
\item if $x,y\in R$ such that $x\pi=y\pi$ then the representatives $x',y'\in R$
of the pullbacks to $R$ of $x\pi$ and $y\pi$ given by the effective 
$\pi\in \Omega$, satisfy $x'\equiv y'\pmod{pR}$.  Each evaluation or 
computation of preimages of $\pi$ uses $O(\rank^3 R)$ operations.
\end{enumerate}
\end{thm}
\begin{proof}
\emph{Algorithm}.
Pass to $\bar{R}:=R/pR\subseteq \End \bar{V}$, $\bar{V}=V/pV$, and using 
\cite[Corollary 1.5]{Ivanyos:fast-alge} compute a Wedderburn
complement decomposition $\bar{R}=\bar{S}\oplus J(\bar{R})$, where 
$\bar{S}$ is a subring of $\bar{R}$ and $\bar{S}\cong \bar{R}/J(\bar{R})$ as rings (note that the direct decomposition is as vector spaces not necessarily as rings).

Apply the \textsf{C-MeatAxe}, \cite{C-MeatAxe}, to $\bar{S}$ 
to find a the set $\mathcal{X}$ of irreducible $\bar{S}$-submodules 
of $\bar{V}:=V/pV$.  As $\bar{S}$ is semisimple, $\mathcal{X}$ is
a direct decomposition of $\bar{V}$.  Conjugate $R$ by the change of
basis matrix resulting from the basis exhibiting the submodules in
$\mathcal{X}$ so that $R$ is block lower triangular.  Use a greedy algorithm to find a minimal subset $\mathcal{W}$ of $\mathcal{X}$ such that $\bar{S}$ acts faithfully on $\langle \mathcal{W}\rangle$.  Let $\tau:\bar{R}\to \bar{S}$ be the projection of $\bar{x}\in\bar{R}$ to $\bar{S}$ given by the vector space decomposition $\bar{R}=\bar{S}\oplus J(\bar{R})$. Pull-backs of $\tau$ are defined by means of the image of basis elements and the linearity of $\tau$.

For each $\bar{W}\in\mathcal{W}$, define $\pi_{\bar{W}}:R\to \End \bar{W}$ by $x\pi_{\bar{W}} := (x+pR)\tau|_{\bar{W}}$, for $x\in R$.  The 
coset representative of the inverse image of $\bar{t}\in \End \bar{W}$ 
is created by extending $\bar{t}$ to $V$ as $\bar{s}$ acting as $0$ on each $\bar{V}_i\neq \bar{W}$, $1\leq i\leq l$ (i.e., $\bar{s}$ has $\bar{t}$ in the $\bar{W}$ diagonal block of the matrix and $0$'s elsewhere), and then returning a coset representative of $\bar{s}\tau^{-1}$.  Thus $\pi$ is an effective homomorphism.  The algorithm returns the set $\{\pi_{\bar{W}} : \bar{W}\in\mathcal{W}\}$.  

\emph{Correctness}. If $M$ is a maximal ideal of $R$ then
$R/M\cong \End_K W$ for some field extension $K/\mathbb{Z}_p$ and
 $K$-vector space $W$.  Hence, $R/M$ is a $\mathbb{Z}_p$-vector space
and so $R/J(R)$ is a $\mathbb{Z}_p$-vector space, which proves that
$pR\leq J(R)$ and $J(\bar{R})=J(R)/pR$.  Therefore, it suffices
to find the projections of $\bar{R}$ onto its simple factors.

Since $R/pR\subseteq \End \bar{V}$ we can apply 
\cite[Corollary 1.5]{Ivanyos:fast-alge}.  
Hence, we obtain a Wedderburn complement decomposition
$\bar{R}=\bar{S}\oplus J(\bar{R})$.  As $\bar{S}$ is semisimple its
action on $\bar{V}$ is completely reducible and the \textsf{C-MeatAxe}
\cite{C-MeatAxe} finds a decomposition $\bar{V}=\bar{V}_1\oplus\cdots\oplus \bar{V}_l$ as above. For each $\bar{W}\in\mathcal{W}$, the map $\pi_{\bar{W}}$ is a ring homomorphism as $\tau$ is a ring homomorphism and $\bar{W}$ is an $S$-module.  Since $\bar{W}$ is
also irreducible it follows that $\bar{T}:=R\pi_{\bar{W}}\leq \bar{S}$ 
is a simple subring of $\End_{\mathbb{Z}_p} \bar{W}$.  The appropriate 
field of scalars is the center $K$ of $\bar{T}$.  Thus $\bar{W}$ is a 
$K$-vector space and $\pi_{\bar{W}}$ is a ring epimorphism
onto $\End_K \bar{W}$ with kernel a maximal ideal of $R$, proving $(i)$.  
Since $\mathcal{W}$ is minimal with respect to having $\bar{S}$ represented 
faithfully on $\langle \mathcal{W}\rangle$, the returned set of epimorphism 
has one epimorphism for each maximal ideal of $R$, thus proving $(ii)$.

Finally, for $(iii)$ we note that the representative matrix for the  inverse image 
under $\pi\in \Omega$, of a point in $\End_K \bar{W}$ 
is trivial in every block except the block on which $\pi$ is projected.  Furthermore,
to evaluate $\pi$ requires we compute $(x+pR)\tau$ which is done by writing
$x+pR$ in the bases of the block decomposition given by $\{V_1,\dots,V_l\}$
and uses $O(\dim^3 \bar{V})$ operations.  To compute a preimage of $\bar{t}$
under $\pi$ requires we write $\bar{t}$ in the basis $\mathcal{X}\tau$
where $\mathcal{X}$ is the fixed basis of $R$.  Therefore the algorithm returns
correctly.

\emph{Timing}. The significant tasks are computing the Wedderburn decomposition
and the use of the \textsf{C-MeatAxe}, both which use $O(\dim^5 \bar{V})$ 
operations in $\mathbb{Z}_p$; see
\cite[Corollary 1.4]{Ivanyos:fast-alge}, \cite{C-MeatAxe}.
\end{proof}

\begin{lemma}\label{lem:pullback}
There is a deterministic polynomial-time algorithm which,
given a $*$-ring epimorphism $\gamma:(R,*)\to (T,*)$
and $t\in T$ such that $t^*=t$, returns an $s\in R$ such that 
$s\gamma=t$ and $s^*=s$.
\end{lemma}
\begin{proof}
\emph{Algorithm}.
Use a basis $\mathcal{X}$ for $\mathfrak{H}(R,*)$ and write
$t=\sum_{x\in\mathcal{X}} s_x x\gamma$, with $s_x\in \mathbb{Z}_{p^e}$.
Return $\sum_{x\in\mathcal{X}} s_x s$.

\emph{Correctness}.  Since $\gamma$ is an $*$-ring homomorphism,
$x\gamma^*=x^*\gamma=x\gamma$.  As $\gamma$ is an epimorphism,
$\mathcal{X}\gamma$ spans the submodule of self-adjoint elements 
of $(T,*)$.  Therefore, $t=\sum_{x\in\mathcal{X}} s_x x\gamma
=\left(\sum_{x\in\mathcal{X}} s_x x\right)\gamma$.  So the
return is correct.

\emph{Timing}.  Assuming a basis for $\mathfrak{H}(R,*)$ is provided,
the task required $O(|\mathcal{X}|)$ evaluations of $\gamma$, and
Gaussian elimination to write $t$ as a linear combination of
$\mathcal{X}\gamma$, which uses $O(|\mathcal{X}|^3)$ operations
in $T$.
\end{proof}

\begin{coro}\label{coro:*-simple-factors}
There is a Las Vegas polynomial-time algorithm which,
given a $*$-ring $(R,*)$ where $R\subseteq \End V$ for an abelian $p$-group $V$,
returns a set $\Gamma=\{\gamma:(R,*)\to (T,*)\}$ of 
$*$-ring epimorphisms where:
\begin{enumerate}[(i)]
\item there is exactly one $\gamma\in \Gamma$ for each maximal $*$-ideal
$M$ of $(R,*)$, and $\ker \gamma=M$.
\item for each $\gamma:(R,*)\to (T,*)\in \Gamma$ either:
\begin{enumerate}[(a)]
\item $T=(M_m(K)\oplus M_n(K),(X,Y)\mapsto (Y^t, X^t))$, or
\item $T=\Adj(d)$ for a nondegenerate symmetric, alternating, or Hermitian form 
$d:K^m\times K^m\to K$.
\end{enumerate}
\item If $x,y\in (R,*)$ such that $x\gamma=y\gamma$ then the representatives $x',y'\in (R,*)$ of the pullbacks to $(R,*)$ of $x\gamma$ and $y\gamma$ given by the effective homomorphism $\gamma\in \Gamma$, satisfy $x'\equiv y'\pmod{pR}$.  Furthermore, if
$x\in \mathfrak{H}(R,*)$ then $x'\in\mathfrak{H}(R,*)$.
\end{enumerate}
\end{coro}
\begin{proof}
\emph{Algorithm}.
Let $\Gamma=\emptyset$.  Using the algorithm of 
\thmref{thm:simple-factors}, compute a representative set of ring epimorphisms 
$\Omega=\{\pi:R\to M_n(K)\}$ corresponding to the maximal ideals of 
$R$.  Take $\pi\in\Omega$ and set $M:=\ker \pi$.  Test if $M^*=M$.  
If so then apply the algorithm of \propref{prop:intoAdj} to construct an effective isomorphism $\varphi:(M_n(K),*)\to \Adj(d)$.  Add $\varphi$ to $\Gamma$ and continue.  Otherwise, find $\pi'\in \Omega$ 
where $\ker \pi'=M^*$.  Then remove $\pi'$ from $\Omega$ 
and define $\gamma:R\to (M_n(K)\oplus M_n(K),\bullet)$ by
$r\gamma:=(r\pi,(r^*\pi')^t)$.  Add $\gamma$ to $\Gamma$ and continue.

\emph{Correctness}.
By \thmref{thm:*-coord}, \propref{prop:intoAdj}, and \thmref{thm:simple-factors} the algorithm returns correctly.

\emph{Timing}.  The number of operations is dominated by the
algorithm for \thmref{thm:simple-factors}.
\end{proof}

\subsection{Finding self-adjoint frames}\label{sec:find-frame}
Let $(R,*)$ be a finite $*$-ring.  We outline how to find
a self-adjoint frame of $\mathfrak{H}(R,*)=\{r\in R : r^*=r\}$.  To do this
we require the following lemma:
\begin{lemma}[Lifting idempotents]\label{lem:lift-idemp}
Suppose that $e\in R$ such that $e^2-e\in \rad R$ and
$(e^2-e)^n=0$ for some $n\in\mathbb{Z}$.  Setting
\begin{equation}\label{eq:lift-idemp}
	\hat{e} := e^n \sum_{j=0}^{n-1} \binom{2n-1}{j} e^{n-1-j} (1-e)^j
\end{equation}
it follows that:
\begin{enumerate}[(i)]
\item $\hat{e}^2=\hat{e}$,
\item $e\equiv \hat{e} \pmod{\rad R}$,
\item $\widehat{1-e}=1-\hat{e}$, and
\item If $*$ is an involution on $R$ and $e^*=e$ then $\hat{e}^*=\hat{e}$.
\end{enumerate}
\end{lemma}
\begin{proof}
$(i)$ through $(iii)$ can be verified directly, compare \cite[(6.7)]{Curtis-Reiner}.  For $(iv)$ notice that $\hat{e}$ is a polynomial in $\mathbb{Z}[e]$.  
As $1^*=1$ and $e^*=e$ it follows that $\hat{e}^*=\hat{e}$.
\end{proof}

\begin{prop}\label{prop:simple-frame}
\begin{enumerate}[(i)]
\item There is a deterministic polynomial-time algorithm which,
given $\Adj(d)$ for a nondegenerate symmetric, alternating, or 
Hermitian form $d:K^n\times K^n\to K$, returns a self-adjoint 
frame of $\Adj(d)$ of maximum possible size.
\item  If $(M_n(K)\oplus M_n(K),\bullet)$
a simple $*$-ring with a standard exchange involution, 
then $\mathcal{E}=\{(E_{ii},E_{ii}): 1\leq i\leq n\}$ is a self-adjoint
frame of $(M_n(K)\oplus M_n(K),\bullet)$ of maximum possible size.
\end{enumerate}
\end{prop}
\begin{proof}
$(i)$. 
\emph{Algorithm}.  If $d$ is alternating, compute a hyperbolic basis
$\mathcal{X}$ for $d$.  If $d$ is symmetric (and non-alternating
if $K$ has characteristic $2$) or Hermitian, then compute an 
orthogonal basis $\mathcal{X}$ for $d$.  Return 
$\mathcal{E}(\{\langle x\rangle: x\in \mathcal{X}\})$.

\emph{Correctness}.  
By \corref{coro:perp-idemp}
we know that the set of frames of $\Sym(d)$ is in bijection with the 
fully refined $\perp$-decompositions of $d$.  As $d$ is a classical form the fully refined $\perp$-decomposition of $d$ are parameterized by standard bases; i.e. a bases $\mathcal{X}$ of $d$ such that for each $x\in \mathcal{X}$ there is a unique $y\in \mathcal{X}$ such that
$d(x,y)\neq 0$.  If $d$ is alternating, this makes $\mathcal{X}$
a hyperbolic basis, and any two hyperbolic bases of $d$ have the
same size $2m$, where $m$ is the Witt index of $d$ (and also the
size of a maximal $\perp$-decomposition of $d$).  If $d$ is 
symmetric not in characteristic $2$, or Hermitian in any characteristic, then $d$ has an orthogonal basis.  The resulting $\perp$-decomposition of $d$ has maximum possible size.  Finally, if $K$ has characteristic $2$ and $d$ is symmetric but non-alternating, then $d$ has an orthogonal basis, and that produces a $\perp$-decomposition of maximum possible size.

\emph{Timing}.  Finding a standard basis of $d$ can be done by 
standard linear algebra at a cost of $O(n^3)$ operations in $K$.

$(ii)$.  Fix $1\leq i\leq n$.  Clearly $(E_{ii},E_{ii})^{\bullet}=(E_{ii}^t,E_{ii}^t)=(E_{ii},E_{ii})=(E_{ii},E_{ii})^2$
$(E_{ii},E_{ii})$ is a self-adjoint idempotent. The 
proper idempotents of $(E_{ii},E_{ii})M_n(K)\oplus M_n(K)(E_{ii},E_{ii})\cong K\oplus K$ are $(E_{ii},0)$ and $(0,E_{ii})$, neither of which is self-adjoint.  Thus, $(E_{ii},E_{ii})$ is a self-adjoint-primitive idempotent.
\end{proof}

\begin{thm}\label{thm:find-frame}
There is a Las Vegas polynomial-time algorithm which, 
given a $*$-ring $(R,*)$ with $R\leq \End V$, $V$ an abelian $p$-group, returns a self-adjoint frame of $(R,*)$ of maximum possible size.
\end{thm}
\begin{proof}
\emph{Algorithm.}
Use the algorithm for \corref{coro:*-simple-factors} to compute a set $\Gamma$ of $*$-epimorphisms onto simple $*$-algebras, one for each maximal $*$-ideal of $(R,*)$.  For each $\gamma:(R,*)\to (T,*)\in \Gamma$, use the algorithm for \propref{prop:simple-frame} to compute a self-adjoint frame $\mathcal{E}_{\gamma}$ of $(T,*)$ of maximum possible size.  Use the algorithm for \corref{coro:*-simple-factors}.$(iii)$ to pullback $\mathcal{E}_{\gamma}$ to a set
\begin{equation*}
	\mathcal{F}_{\gamma}=\{f\in R: f^2\equiv f \pmod{pR}, 
		f^*\equiv f \pmod{pR}\},
\end{equation*}
with $\mathcal{F}_{\gamma}\gamma=\mathcal{E}_{\gamma}$.  Apply \eqref{eq:lift-idemp} to the members of $\mathcal{F}_{\gamma}$ to create $\mathcal{G}=\{\hat{f} : f\in\mathcal{F}_{\gamma},\gamma\in \Gamma\}$.  Return $\mathcal{G}$.

\emph{Correctness}.
Evidently, $\mathcal{E}:=\sqcup_{\gamma\in\Gamma} \mathcal{E}_{\gamma}$ is a self-adjoint frame of $(R/J(R),*)$ of maximum possible size.  The pullback $\mathcal{F}:=\sqcup_{\gamma\in\Gamma}\mathcal{F}_{\gamma}$ consists of self-adjoint elements of $(R,*)$ for which $\mathcal{F}\gamma=\mathcal{E}$ and the two sets are in bijection. By \lemref{lem:lift-idemp}, the return $\mathcal{G}$ is a self-adjoint frame of $(R,*)$ of maximum possible size.

\emph{Timing.}  The algorithm for \corref{coro:*-simple-factors} uses $O(\rank^5 V)$ operations $\mathbb{Z}_{p^e}$.  Fix $\gamma:(R,*)\to(T_{\gamma},*)\in \Gamma$ with $T_{\gamma}=\End_K W_{\gamma}$.  \propref{prop:simple-frame} uses $O(\rank^3 W_{\gamma})$ operations in $K_{\gamma}$; thus, $O(\log^3 W_{\gamma})$ operations in $\mathbb{Z}_p$.  Since $\sum_{\gamma\in\Gamma} \rank W_{\gamma}$ is at most $\rank V$, it follows that this stage takes at most $O(\log^3 |V|)$ operations in $\mathbb{Z}_p$.  

The algorithms for \lemref{lem:pullback} uses $O(\rank^3 T_{\gamma})$ operations in $\mathbb{Z}_{p^e}$.  Since the bases computed in \lemref{lem:pullback} can be reused for each application with respect to a fixed $\gamma$, it follows that the total number of operations in $\mathbb{Z}_{p^e}$ uses $$O\left(\sum_{\gamma\in \Gamma}\rank^3 T_{\gamma}\right)=O\left(\sum_{\gamma\in \Gamma} \rank^6 W_{\gamma}\right)=O(\log^6 |V|)$$ operations in $\mathbb{Z}_{p^e}$.
\end{proof}

\section{Proof of Theorem \ref{thm:main}}\label{sec:final}
\begin{proof}[Proof of \thmref{thm:main}]
\noindent\emph{Algorithm}.
Given a finite $p$-group $P$ of class $2$, compute bases for $P/Z(P)$
and $P'$ and compute a structure constant
representation of $b:=\Bi(P)$ (which is straightforward from
the definitions in Section \ref{sec:bi-grp} and \eqref{eq:bilinear-comp}).

Next, compute a basis for $\Adj(b)$ (Section \ref{sec:adj-sym-algo}).  Apply \thmref{thm:find-frame} to find a self-adjoint frame $\mathcal{E}$ of $\Adj(b)$ of maximum possible size.  Induce a fully refined $\perp$-decomposition $\mathcal{V}=\{(P/Z(P))e:e\in\mathcal{E}\}$ of $b$ (cf. \corref{coro:perp-idemp}).

Apply \corref{coro:convert} to produce a fully refined 
central decomposition of $P$.  

\emph{Correctness}.  This follows from \corref{coro:convert},
\corref{coro:perp-idemp}, and \thmref{thm:find-frame}.

\emph{Timing}.
Since $\rank \Adj(b)\leq \log_p^2 [P:Z(P)]^2\leq \log^2 [P:P']$, 
the total number of operations in $\mathbb{Z}_{p^e}$ lies in
$O(\log^6 [P:P'])$.

\emph{Deterministic version}
Suppose that $p$ is small ($p\leq\log^c |P|$ for some constant $c$).  
Here, the Las Vegas method of \cite{Ivanyos:fast-alge} can be replaced
by the deterministic methods of \cite{Ronyai} in the algorithm
of \thmref{thm:simple-factors}.  Consequently, every Las Vegas algorithm is replaced by a deterministic algorithm.
\end{proof}

\subsection{Bottlenecks}\label{sec:bottlenecks}

The main bottleneck in practice is computing generators for 
$\Adj(\Bi(P))$ for a given $p$-group $P$.  Examples carried with in collaboration with P. A. Brooksbank \cite{Brooksbank-Wilson:*-ring,Brooksbank-Wilson:isom} show that with a group of size $p^{40}$, for $p\in\{5,7,11\}$, a conventional laptop used roughly 5 seconds of real-time to compute generators for $\Adj(\Bi(P))$ and only milliseconds to determine the $*$-ring structure of $\Adj(\Bi(P))$.  Sometimes this occurs because the rank of $\Adj(\Bi(P))$ can be small as compared to the rank of $P$.  However, examples of groups of order $p^{196}$ with intentionally large adjoint $*$-rings with radicals and multiple $*$-simple factors still spend most of the time computing generators for $\Adj(\Bi(P))$, roughly 1 hour as compared to the 1 minute spent in identify the ring structure.  For details see \cite{Brooksbank-Wilson:isom}.

\section{Related results}\label{sec:closing}

We summarize some of the related applications of the algorithm
and methods for \thmref{thm:main}.

\subsection{Central decompositions of Lie rings}

There is related problem of central decomposition of nilpotent Lie
ring $L$ of class $2$; see \cite[p.608-609]{Bond}.  Though we do not 
require $L$ be an algebra over a field, we assume that multiplication 
in $L$ is $K$-bilinear for some commutative ring $K$ (not necessarily
finite or of positive characteristic) for which computation is feasible either in polynomial-time or tolerable in practice, so we call $L$ a Lie $K$-algebra.  $L$ should be specified by reasonable means, for instance, generated by matrices 
under the usual commutator bracket, or given with a basis and structure
constants.  

\begin{thm}\label{thm:Lie}
Suppose that $K$ is a commutative local ring with an oracle to factor
polynomials in $K[x]$.  Then, there is a Las Vegas polynomial-time 
algorithm which, given a finite rank nilpotent Lie $K$-algebra of 
class $2$, returns a central decomposition of $L$ of maximum size.
\end{thm}
\begin{proof}
\emph{Algorithm}.  
Define $\Bi(L):L/Z(L)\times L/Z(L)\to [L,L]$ by 
$\Bi(L)(Z(L)+x,Z(L)+y):=[x,y]$, for all $x,y\in L$.
Compute $\Adj(\Bi(L))$ and use \thmref{thm:find-frame} to find
a self-adjoint frame $\mathcal{E}$ of $\Bi(L)$ (which requires the
polynomial factorization oracle \cite[Section 4,5]{Ronyai}).  
Pullback the decomposition to $L$ and apply the algorithms
for \propref{prop:reduce} and \corref{coro:convert} using 
$\Phi(L):=J(K)L+[L,L]$ in the r\^{o}le of $\Phi(P)$.

\emph{Correctness}.  The proof is the same as \thmref{thm:main}.

\emph{Timing}.  The overall number of operations spent in computing 
$\Adj(b)$ and in \thmref{thm:find-frame} which both lie in 
$O(\rank^6 L)$.
\end{proof}

\begin{remark}
The practicality of \thmref{thm:Lie} depends on the practicality
of the oracle to factor polynomials and working in $K$.  Over $\mathbb{Q}$, factoring polynomials is as difficult as factoring integers and therefore not a polynomial-time process.  However, in practice that ``glitch'' is often of little distress.
\end{remark}

\subsection{Determining the types of centrally indecomposables}

\begin{thm}\label{thm:types}
A $p$-group $P$ of class $2$ is centrally indecomposable 
if, and only if, $Z(P)\leq \Phi(P)$ and $\Adj(\Bi(P))/J(\Adj(\Bi(P)))$
is $*$-isomorphic to one of the following:
\begin{enumerate}
\item \textsf{Orthogonal type}: $GF(p^e)$ with identity involution,
\item \textsf{Unitary type}:
$GF(p^e)$, $2|e$,  with field involution of order $2$,
\item \textsf{Exchange type}:
$GF(p^e)\oplus GF(p^e)$ with involution $(x,y)^*:=(y,x)$,
for all $x,y\in GF(p^e)$; or
\item \textsf{Symplectic type}: $M_2(GF(p^e))$ with involution
\begin{equation}
\begin{bmatrix} a & b \\ c & d\end{bmatrix}^*
	:= \begin{bmatrix} d & -b \\ -c & a\end{bmatrix},
	\qquad \forall a,b,c,d\in GF(p^e).
\end{equation}
\end{enumerate}
\end{thm}
\begin{proof}
By \thmref{thm:cent-indecomp} it remains to show that $(S,*):=\Adj(\Bi(P))/J(\Adj(\Bi(P)))$ is one of the algebras listed.  By \corref{coro:perp-idemp}, we know $(S,*)$ has no proper self-adjoint-primitive idempotents.  This makes $(S,*)$ a $*$-simple ring. 

If $S$ is classical, then
by \propref{prop:intoAdj} it follows that $S$ is $*$-isomorphic to
$\Adj(d)$ for a nondegenerate symmetric, alternating, or Hermitian form 
$d:K^n\times K^n\to K$.  By \corref{coro:perp-idemp}, $d$ must be 
$\perp$-indecomposable, so $n=1$ if $d$ is symmetric or
Hermitian, or $n=2$ if $d$ is alternating.  This handles cases
$(1)$, $(2)$, and $(4)$.  

If $S$ has an exchange involution, then by 
\propref{prop:simple-frame}.$(ii)$,
$S$ must be $*$-ring isomorphic to the $*$-ring in $(3)$.
\end{proof}

\begin{remark}
Examples of the orthogonal, exchange, and symplectic type 
were first given in \cite[Section 7]{Wilson:unique}. 
Appendix \ref{app:simples} includes new examples including the first
examples of unitary type.
\end{remark}

\begin{coro}
There is polynomial-time algorithm which, given a 
finite centrally indecomposable $p$-group of class 2, returns
the type of the group as listed in \thmref{thm:types}.
\end{coro}
\begin{proof}
This is immediate from \thmref{thm:simple-factors}, 
\propref{prop:intoAdj}, and \thmref{thm:types}.
\end{proof}

\subsection{Testing indecomposability}

Suppose that we are only interested in testing if a $p$-group $P$ 
of class $2$ is centrally indecomposable.  By \thmref{thm:types}, 
the key step is to determine that $\Adj(b)/J(\Adj(b))$ is 
one of the four algebras in that list.  That process is easier 
in the present framework as it requires that there be at most 2 isomorphism types of simple modules in the composition series of $V$ as an $\Adj(b)$-module.  Furthermore, the simple modules have dimension 1 or 2 when viewed over the correct field, i.e. $Z(\Adj(b)/J(\Adj(b)))$.  This can be determined using the absolute irreducibility test of the MeatAxe \cite{MeatAxe1},  thus reducing the time in those stages to $O(\log^4 |V|)$-time.  Unfortunately, the bottleneck
remains in computing generators for $\Adj(b)$, which still requires
$O(\log^5 |V|)$-time.

\subsection{Finding orbits of central decompositions}
In \cite{Wilson:unique}, the action of the automorphism group of
a $p$-group $P$ of class $2$ and exponent $p$ was studied.  
Though not presented in detail, it is clear
that the methods here can be used to find a representative fully refined central
decomposition for each $C_{\Aut P}(Z(P))$-orbit as described in \cite[Corollary 5.23.(iii)]{Wilson:unique}.  The necessary step is to choose an orthogonal
basis in \propref{prop:simple-frame} with the desired address in the sense of
\cite[Definition 5.1]{Wilson:unique}.  

\subsection{Finding some new characteristic and fully invariant subgroups}

We now show how the $*$-ring $\Adj(\Bi(P))$ can be used to uncover
new characteristic and fully invariant subgroups of $P$.

Recall that $\Adj(\Bi(P))$ is a subring of $\End V\times (\End V)^{op}$
where $V:=P/Z(P)$.  Thus, $\Adj(\Bi(P))$ acts on $V$ by
$v(f,g):=vf$, for all $v\in V$ and all $(f,g)\in \Adj(\Bi(P))$.  
If $I$ is a right ideal of $\Adj(\Bi(P))$ then $VI$ is a submodule of $V$.
Recall that an ideal $I$ of a $*$-ring $R$ is \emph{$*$-characteristic}
(\emph{$*$-fully invariant}) if 
$I\varphi=I$ for all $*$-ring automorphisms (endomorphisms) of $R$.
We prove:
\begin{thm}\label{thm:char}
For a $p$-group $P$ of class $2$, 
\begin{equation}
	\mathcal{L}:=\{Z(P)\leq L\leq P : L/Z(P)=(P/Z(P))I,
	\textnormal{ I $*$-characteristic in }\Adj(\Bi(P))\}
\end{equation}
is a lattice of characteristic subgroups of $P$.
\end{thm}
\begin{proof}
$\Aut P$ acts on $\Adj(\Bi(P))$ via
\begin{equation}
	(f,g)^{\varphi}:=(\varphi|_V^{-1} f\varphi|_V,
		\varphi|_V g\varphi_{V}^{-1}),\qquad
		\forall (f,g)\in \Adj(\Bi(P)), \varphi\in \Aut P.
\end{equation}
That action commutes with the $*$ involution on $\Adj(\Bi(P))$;
so every $*$-characteristic $*$-ideal $I$ of $\Adj(\Bi(P))$ is acted
on by $\Aut P$.  Thus,
$0\leq VI\leq V$ is an $\Aut P$-submodule of $V=P/Z(P)$. 
As $Z(P)$ is characteristic in $P$, pulling back to $P$ proves our claim.
\end{proof}

\begin{remark}\label{rem:full}
There is a bilinear map $\Bi(P,P')$ from $P/P'\times P/P'\to P'$ 
defined analogously to $\Bi(P)$.  This bilinear map 
may be degenerate; thus, $\Adj(\Bi(P,P'))$ is not necessarily a 
$*$-ring.  However, because $P'$ is fully invariant, it follows that
\begin{equation}
	\mathcal{L}:=\{P'\leq L\leq P : L/P'=(P/P')I,
	\textnormal{ I fully invariant in }\Adj(\Bi(P,P'))\}
\end{equation}
is a lattice of fully invariant subgroups of $P$.
\end{remark}

Using the radical and semisimple structure of $\Adj(\Bi(P))$ it is 
easy to identify various specific $*$-characteristic and
$*$-fully invariant $*$-ideals of $\Adj(\Bi(P))$.
\begin{ex}\label{ex:char}  Given a $*$-ring $(R,*)$:
\begin{enumerate}[(i)]
\item if $J$ is the Jacobson radical of $R$, then
$\{J^i:i\in\mathbb{Z}^+\}$ is a flag of $*$-fully invariant $*$-ideals
of $(R,*)$; and
\item the intersection of all maximal $*$-ideals with $*$-ring 
isomorphic quotients is a $*$-fully invariant $*$-ideal of $(R,*)$.
\end{enumerate}
\end{ex}

\begin{coro}
There are polynomial-time algorithms which, given a $p$-group $P$ of
class $2$, return the characteristic and fully invariant 
subgroups of $P$ resulting from \exref{ex:char} and \thmref{thm:char}
or \remref{rem:full}.
\end{coro}
\begin{proof}
This is an obvious application of the $*$-ring structure algorithms
given in Section \ref{sec:*-rings}.
\end{proof}
  
\section{Central products of general groups}\label{sec:general}

We deviate from our focus on $p$-groups of class $2$ to address some
of the situation for central decompositions of general finite groups.

\subsection{Central products and isoclinism}
We apply an equivalence relation on groups introduced by P. Hall
\cite{Hall:isoclinism} which is compatible with 
central products.  This allows for a partial generalization of
the concepts in Section \ref{sec:groups}.

The proof of \thmref{thm:main} concentrates on the
bilinear map $\Bi(P):P/Z(P)\times P/Z(P)\to P'$.
Evidently nonisomorphic $p$-groups can have equivalent
bilinear maps.  Equivalence of bilinear maps $b:V\times V\to W$
and $b':V'\times V'\to W'$ is defined by pairs of linear
maps $(f:V\to V',\hat{f}:W\to W')$ such that:
\begin{equation}
	b'(uf,vf) = b(u,v)\hat{f},\qquad\forall u,v\in V.
\end{equation}
More generally, an \emph{isoclinism} \cite{Hall:isoclinism} 
of groups $G$ and $H$ is a pair 
$(\alpha:G/Z(G)\to H/Z(H),\hat{\alpha}:G'\to H')$ of
group isomorphisms such that
\begin{equation}
	[Z(G)x\alpha,Z(G)y\alpha]=[x,y]\hat{\alpha},
			\qquad \forall x,y\in G.
\end{equation}
Isomorphic groups are immediately isoclinic, but the converse is
false (abelian groups are isoclinic to the trivial group).  
Clearly, $\Bi(P)$ and $\Adj(\Bi(P))$ are group isoclinism invariants of $P$. Moreover, if $G$ and $H$ are general groups and
$\mathcal{K}$ is a central decomposition of $G$, then 
\begin{equation}
	\mathcal{K}\alpha:=\{Z(H)\leq J\leq H:\exists K\in\mathcal{K},
		J/Z(H)=KZ(G)/Z(G)\alpha\}
\end{equation}
is a central decomposition of $H$.  Call a central decomposition
$\mathcal{H}$ of $G$ a $Z(G)$-central decomposition if 
$\mathcal{H}=\mathcal{H}Z(G)$.  Thus we have proved:
\begin{prop}
An isoclinism from a group $G$ to a group $H$ induces a bijection
from the set of $Z(G)$-central decompositions of $G$ and the set of 
$Z(H)$-central decompositions of $H$.  In particular, if $G$ is
centrally indecomposable, then every central decomposition of $H$
has at most one nonabelian member.
\end{prop}

Using group isoclinism we can generalize a conjecture made in
\cite{Wilson:unique}.

Examples such as $D_8\circ D_8\cong Q_8\circ Q_8$ and the similar
problem for odd extraspecial groups  of exponent $p^2$ (see \cite[Theorem 5.5.2]{Gor}) demonstrate that the group isomorphism classes of 
a central decomposition of \emph{maximum} possible size need not be 
the same.  However, we ask:
\begin{quote}
Is the multiset of group isoclinism types of a central decomposition
of maximum possible size a group isoclinism invariant?
\end{quote}
We conjecture that this is true for $p$-groups of class $2$.  If so, then it is probably true for all groups; in particular, the problem for nilpotent groups of larger class and groups with no center can benefit from the uniqueness afforded by the Krull-Remak-Schmidt theorem; compare \cite[Section 4.3.4]{Wilson:thesis}.

\subsection{Idempotents in central products of general groups}
\label{sec:gen}
Let $G$ be any group.
We can define $c:=c(G):G/Z(G)\times G/Z(G)\to G'$ by
\begin{equation}
	c(Z(G)x,Z(G)y):=[x,y],\qquad \forall x,y\in G.
\end{equation}
We also have:
\begin{equation}\label{eq:adj2}
\begin{split}
	\Adj(c) := & \{(f,g)\in \End G/Z(G)\times (\End G/Z(G))^{op} :\\
		& c(Z(G)xf,Z(G)y)=c(Z(G)x,Z(G)yg),\forall x,y\in G\}.
\end{split}
\end{equation}
Obviously, $\Adj(c)$ is closed to products and has an anti-automorphism 
$*:(f,g)\mapsto (g,f)$ of order $2$.  However, unlike $\Adj(\Bi(P))$, 
$\Adj(G)$ need not be a ring since we cannot generally add endomorphisms 
of $G/Z(G)$.  Nonetheless, it follows that:
\begin{prop}
The set of $Z(G)$-central decompositions of $G$ is in bijection with
the set of sets of self-adjoint idempotents of $\Adj(c(G))$.
\end{prop}
\begin{proof}
The proof is the same as that of \propref{prop:cent-perp} 
and \corref{coro:perp-idemp}.
\end{proof}

\subsection{Finding central products of general groups}
To find central decompositions of
 groups $G$ which are either non-nilpotent or nilpotent of class
greater than $2$, it is may be possible to begin by finding
direct decompositions of $G/Z(G)$, and then reduce to 
central decomposition of $G$.  The first polynomial-time algorithm 
to find a direct product decomposition of a finite 
group appeared in \cite[Chapter IV]{Wilson:thesis} along with the
algorithms of \thmref{thm:main} \cite[Chapter III]{Wilson:thesis}. 
 A preliminary inspection
supports the conjecture that a combination of these two results will
produce a polynomial-time algorithm to find fully refined central 
decompositions of arbitrary finite groups.

\section*{Acknowledgments}
Considerable thanks belong to W. M. Kantor for patiently listening
to the creation of these algorithms and encouraging their
development.  Thanks also to E. M. Luks and C.R.B. Wright for 
their comments on earlier versions; to L. Ronyai for 
discussing the state of algorithms for associative rings; to the anonymous referee for suggesting improvements; and finally to 
P. A. Brooksbank for discussions which lead to improvements and
implementation of key components of these algorithms in 
\textsf{MAGMA} \cite{Magma}.

\appendix

\section{Examples of centrally indecomposable groups}\label{app:simples}

We give examples which demonstrate some of the important
aspects of the algorithm of \thmref{thm:main}.  As evidence that all cases considered here can occur, we give $p$-groups which are centrally indecomposable of each of the types listed in \thmref{thm:types}.  Furthermore, our proofs apply the techniques of \thmref{thm:main} in a symbolic fashion illustrating how the methods can be used beyond a computer.

\begin{ex}[A centrally indecomposable group of orthogonal type]
\begin{equation}\label{eq:free-rank-3}
	Or_p=\langle a,b,c~|~a^p, b^p, c^p, [a,b]^p, [a,c]^p, [b,c]^p,
		 \textnormal{ class 2 }\rangle.
\end{equation}
is a special group of order $p^6$ and is centrally indecomposable
of orthogonal type.
\end{ex}
\begin{proof}
Let $P:= Or_p$.
Clearly, $P^p\leq P'=Z(P)$ so $P$ is a special $p$-group of
order $p^6$ and rank $3$.  Therefore $P/Z(P)\cong \mathbb{Z}_p^3$
and $P'\cong \mathbb{Z}_p^3$.  So 
$\Bi(P):\mathbb{Z}_p^3\times\mathbb{Z}_p^3\to \mathbb{Z}_p^3$.  Using
$\{Z(P)s,Z(P)t,Z(P)u\}$ and $\{x:=[a,b],y:=[a,c],z:=[b,c]\}$ as 
ordered bases for $P/Z(P)$ and $P'$ respectively, it is evident from
\eqref{eq:free-rank-3} that $\Bi(P)$ is defined by
$\Bi(P)(u,v)=uBv^t$ for all $u,v\in\mathbb{Z}_p^3$, where:
\begin{equation}
	B :=\begin{bmatrix} 0 & x & y \\ -x & 0 & z\\ -y & -z & 0
		\end{bmatrix}.
\end{equation}
Computing $\Adj(B)$ as in Section \ref{sec:adj-sym-algo} (which is
easily done with symbolic computation on an example of this size; 
compare \cite[Lemma 7.1]{Wilson:unique})
we find 
\begin{equation}
	\Adj(B)=\{(\alpha I_3,\alpha I_3) \in 
		M_3(\mathbb{Z}_p)\times M_3(\mathbb{Z}_p):
		\alpha\in\mathbb{Z}_p\}=\Sym(b),
\end{equation}
which is clearly $*$-isomorphic to $\mathbb{Z}_p$ with identity 
involution.  By \thmref{thm:types}, $P$ is centrally indecomposable
of orthogonal type.  
\end{proof}

\begin{ex}[A centrally indecomposable group of exchange type]
\begin{equation}
\begin{split}
E_p:= &\langle a,b,c,d~|~a^p,b^p,c^p,d^p,[a,c]^p,[a,d]^p,[b,c]^p,[b,d]^p,
	[a,b],[c,d],\textnormal{class }2\rangle
\end{split}
\end{equation}
is a special group of order $p^{8}$ and is centrally indecomposable of
exchange type.
\end{ex}
\begin{proof}
$\Bi(E_p)$ is bilinear map $\mathbb{Z}_p^4\times \mathbb{Z}_p^4\to \mathbb{Z}_p^4$.  With respect
to the bases $$\{Z(E_p)a,Z(E_p)b,Z(E_p)c,Z(E_p)d\}$$ and 
$$\{x:=[a,c],y:=[a,d],
z:=[b,c],w:=[b,d]\},$$ 
$\Bi(E_p)$ is defined by
\begin{equation}
	B := \left[\begin{array}{cc|cc} 
		0 & 0 & x & y\\ 
		0 & 0 & z & w\\
		\hline
		-x & -z & 0 & 0\\ 
		-y & -w & 0 &0
		\end{array}
		\right].
\end{equation}
Evidently,
\begin{equation}
	\Adj(B) =\left\{\left(
	 \begin{bmatrix} \alpha & 0 \\ 0 & \beta\end{bmatrix}\otimes I_2,
	 \begin{bmatrix} \beta & 0 \\ 0 & \alpha \end{bmatrix}\otimes I_2\right):
	  \alpha,\beta\in \mathbb{Z}_p
	\right\}.
\end{equation}
This is $*$-ring isomorphic to 
$(\mathbb{Z}_p\oplus \mathbb{Z}_p,(\alpha,\beta)^\bullet=(\beta,\alpha))$.
Thus, $E_p$ is a centrally indecomposable group of exchange type.
\end{proof}

\begin{ex}[A centrally indecomposable group of unitary type]
Let $p$ be odd and $\omega\in\mathbb{Z}$ be a non-square modulo $p$.
\begin{equation}
\begin{split}
U_p:= &\langle a,b,c,d,e,f~|~a^p,b^p,c^p,d^p,e^p,f^p,
	[a,b]^p,[a,c]^p,[b,c]^p,\\
	& [a,b]^{\omega}[d,e],[a,c]^{\omega}[d,f],[b,c]^{\omega}[f,e],\\
	& [a,e],[a,f],[b,d],[b,e],[b,f],[c,d],[c,f],[c,e],\textnormal{class }2\rangle
\end{split}
\end{equation}
is a special group of order $p^{12}$ and is centrally indecomposable of
unitary type.
\end{ex}
\begin{proof}
$\Bi(U_p)$ is bilinear map $\mathbb{Z}_p^6\times \mathbb{Z}_p^6\to \mathbb{Z}_p^4$.  With respect
to the bases $$\{Z(U_p)a,Z(U_p)b,Z(U_p)c,Z(U_p)d,Z(U_p)e,Z(U_p)f\}$$ and 
$\{x:=[a,b],y:=[a,c],z:=[b,c],u:=[a,d]\}$, $\Bi(U_p)$ is defined by
\begin{equation}
	B := \left[\begin{array}{ccc|ccc} 
		0 & x & y & u & 0 & 0\\ 
		-x & 0 & z & 0 & 0 & 0\\
		-y & -z & 0 & 0 & 0 & 0 \\
		\hline
		-u & 0 & 0 & 0 & -\omega x & -\omega y\\ 
		0 & 0 & 0 & \omega x & 0 & -\omega z \\
		0 & 0 & 0 & \omega y & \omega z & 0
		\end{array}
		\right].
\end{equation}
By computing we find:
\begin{equation}
	\Adj(B) =\left\{\left(
	\begin{bmatrix} \alpha & \beta\\ \omega \beta & \alpha\end{bmatrix}
	\otimes I_3, \begin{bmatrix} \alpha & -\beta \\ -\omega\beta
		& \alpha \end{bmatrix}\otimes I_3\right) :
		\alpha,\beta\in \mathbb{Z}_p\right\}.
\end{equation}
This is $*$-ring isomorphic to $GF(p^2)=\mathbb{Z}_p[x]/(x^2-\omega)$ 
with field involution $\sqrt{\omega}\mapsto -\sqrt{\omega}$.
Thus, $U_p$ is a centrally indecomposable group of unitary type.
\end{proof}

\begin{ex}[A centrally indecomposable group of exchange type]
\begin{equation}
\begin{split}
p^{1+2}_{+}:= &\langle a,b~|~a^p,b^p,[a,b]^p,\textnormal{class }2\rangle
\end{split}
\end{equation}
is an extraspecial group of order $p^{3}$ and is centrally 
indecomposable of symplectic type.
\end{ex}
\begin{proof}
$\Bi(p^{1+2}_+)$ is bilinear map $\mathbb{Z}_p^2\times \mathbb{Z}_p^4\to \mathbb{Z}_p^4$.  With respect
to the bases $$\{Z(E_p)a,Z(E_p)b\}\textnormal{ and }\{x:=[a,b]\},$$
 $\Bi(p^{1+2}_+)$
is defined by
\begin{equation}
	B := \begin{bmatrix} 0 & x\\ -x & 0\end{bmatrix}
\end{equation}
Clearly,
\begin{equation}
	\Adj(B) =\left\{\left(
	 \begin{bmatrix} \alpha & \beta \\ \gamma & \delta\end{bmatrix},
	 \begin{bmatrix} \delta & -\beta \\ -\gamma & \alpha \end{bmatrix}
	 \right):	  \alpha,\beta,\gamma,\delta\in \mathbb{Z}_p
	\right\}.
\end{equation}
This is $*$-ring isomorphic to $M_2(\mathbb{Z}_p)$ with symplectic
involution.
Thus, $p^{1+2}_+$ is a centrally indecomposable group of symplectic type.
\end{proof}

\section{Examples of centrally decomposable groups}\label{app:Tang}

We now demonstrate how central products can be used to characterize
$p$-groups.  We conclude by reproving an example of Taft and generalizing it to an infinite family.

The most common central product is one where a group is created as a central product of a single group $G$ with center identified in a natural fashion, specifically: 
$$G\circ G:=G\times G/\langle (x,y)\in 
Z(G\times G): xy=1\rangle.$$
A subtle generalization is to include exponents in the identification:
\begin{defn}
Fix $(a_1,\dots, a_n)\in \mathbb{Z}^n$.
For a group $G$ define:
\begin{equation}
	G^{\circ(a_1,\dots,a_n)}
		= G^n/\langle (x_1,\dots, x_n)\in Z(G^n) :
			x_1^{a_1}\cdots x_n^{a_n}=1\rangle.
\end{equation}
\end{defn}
Evidently, $G^{\circ(a_1,\dots,a_n)}$ is an $n$-fold central product
of $G$ but where the centers are identified according to the given
exponents.

\begin{ex}
Define
\begin{equation}\label{eq:odd-even-ex}
\begin{split}
	R_p  := & \langle a,b,c,d,e,f~|~ a^p,b^p,c^p,d^p,e^p,f^p,\\
		& [a,b][a,e]^{-2},~[a,c][a,f]^{-2},~[a,d],~[a,e]^p,~[a,f]^p,\\
		& [b,c][b,f]^{-2},~[b,d][a,e],~[b,e],~[b,f]^p,\\
		& [c,d][a,f],~[c,e][b,f],~[c,f],\\
		& [d,e][a,e]^{-2},~[d,f][a,f]^{-2},\\
		& [e,f][b,f]^{-2},
		\textnormal{class 2}\rangle.
\end{split}
\end{equation}
Then $|R_p|=p^{9}$, $R'_p\cong \mathbb{Z}_p^3$, $R_p/R'_p\cong \mathbb{Z}_p^6$.
Furthermore,
\begin{enumerate}[(i)]
\item $R_2$ is a special group and centrally indecomposable of symplectic
type,
\item $R_3\cong Or_3\times \mathbb{Z}_3^3$ and is almost special
(i.e.: $R_3'=\Phi(R_3)\leq Z(R_3)$ and $Z(R_3)$ is elementary abelian),
and
\item if $p>3$ then $R_p\cong  Or_p^{\circ(1,3)}$.  Furthermore,
$R_p\cong Or_p\circ Or_p$ if, and only if, $3$ is a square modulo $p$.
\end{enumerate}
\end{ex}
\begin{proof}
Set $x:=[a,e]$, $y:=[a,f]$, and $z:=[b,f]$.  Evidently $R_p'=\langle x,y,z\rangle\cong \mathbb{Z}_p^3$ and $R_p$ has order $p^9$.  

When $p\neq 3$, $R'_p=Z(R_p)$.  Furthermore, 
\begin{equation}
R_p/Z(R_p)=\langle Z(R_p)a,Z(R_p)b,Z(R_p)c,Z(R_p)d,Z(R_p)e,Z(R_p)f\rangle
	\cong \mathbb{Z}_p^6.
\end{equation}
With respect to the given generators, 
$b:=\Bi(R_p):\mathbb{Z}_p^6\times \mathbb{Z}_p^6\to \mathbb{Z}_p^3$ 
(Section \ref{sec:bi-grp}) is defined by $b(u,v)=uBv^t$, for all
$u,v\in\mathbb{Z}_p^6$, where:
\begin{equation}\label{eq:B2}
B:=\left[\begin{array}{ccc|ccc}
0 & 2x & 2y & 0 & x & y\\
-2x & 0 & 2z & -x & 0 & z \\
-2y & -2z & 0 & -y & -z & 0\\
\hline
0 & x  & y & 0 & 2x & 2y \\
-x & 0 & z & -2x & 0 & 2z \\
-y & -z & 0 & -2y & -2z & 0
\end{array}\right]
= \begin{bmatrix} 2 & 1 \\ 1 & 2\end{bmatrix}
\otimes \begin{bmatrix} 0 & x & y \\ -x & 0 & z\\ -y & -z & 0\end{bmatrix}.
\end{equation}
(The tensor notation is the usual Kroncher product and we use it here
to compress the data; for more on adjoints and tensors see 
\cite[Section 7.2]{Wilson:unique}.)  As $p\neq 3$, $D:=\begin{bmatrix}
2 & 1 \\ 1 & 2\end{bmatrix}$ is invertible.  Furthermore, 
computing $\Adj(B)$ as in Section \ref{sec:adj-sym-algo} shows that
\begin{equation}
\Adj(B) = \{ (A\otimes I_3,
D A^t D^{-t}\otimes I_3): A\in M_2(\mathbb{Z}_p)\}
\end{equation}
which is $*$-isomorphic to $\Adj(D)$.  The bilinear map
$d:\mathbb{Z}_p^2\times \mathbb{Z}_p^2\to \mathbb{Z}_p$ given by
$d(u,v):=uDv^t$, for all $u,v\in\mathbb{Z}_p^2$, is a symmetric nondegenerate bilinear form.  

$(i)$. If $p=2$, then $d$ is also a nondegenerate alternating bilinear form
of dimension $2$, and $d$ is $\perp$-indecomposable.  Thus
$\Adj(d)\cong \Adj(D)\cong \Adj(B)$ has no proper self-adjoint-primitive
idempotents; that is, $R_2$ is centrally indecomposable of symplectic
type; see \thmref{thm:types}.

$(iii)$.
If $p>3$ then $d$ has an orthogonal basis, for instance 
$\{(1,-1),(1,1)\}$.  This produces the following self-adjoint frame
for $\Adj(d)$:
\begin{equation}
\left\{
e:=
\begin{bmatrix}
1/2 & -1/2 \\
-1/2 & 1/2
\end{bmatrix},
f:=
\begin{bmatrix}
1/2 & 1/2 \\
1/2 & 1/2
\end{bmatrix}
\right\}.
\end{equation}
Thus, $\mathcal{E}:=\{e\otimes I_3,f\otimes I_3\}$ is a self-adjoint frame of $\Adj(B)$ and decomposes $\mathbb{Z}_p^6$ into:
\begin{eqnarray}
E & := & 
	\langle (1,0,0,-1,0,0),(0,1,0,0,-1,0),(0,0,1,0,0,-1)\rangle,\\
F & := & 
	\langle (1,0,0,1,0,0),(0,1,0,0,1,0),(0,0,1,0,0,1)\rangle,
\end{eqnarray}
and $b(E,F)=0$; thus, $\{E,F\}$ is a fully refined $\perp$-decomposition
of $b$ of maximum possible length.
Pulling back to subgroups of $R_p$ we have:
\begin{equation}
	H_{(1,-1)} = \langle ad^{-1},be^{-1},cf^{-1}\rangle \leq R_p
	\textnormal{ and }
	H_{(1,1)} = \langle ad,be,cf\rangle \leq R_p.
\end{equation}
So $\{H_{(1,-1)},H_{(1,1)}\}$ is a fully refined central decomposition 
of $R_p$.  Indeed, if we change the basis of $b$ so to the bases given 
for $E$ and $F$ we have:
\begin{equation}
\tilde{B}:= \begin{bmatrix} 2 & 0 \\ 0 & 6\end{bmatrix}
\otimes \begin{bmatrix} 0 & x & y \\ -x & 0 & z\\ -y & -z & 0\end{bmatrix}.
\end{equation}
Let $\tilde{x}:=2x=[a,b]$, $\tilde{y}:=2y=[a,c]$, and
$\tilde{z}:=2z=[b,c]$.  Thus, 
\begin{equation}
\tilde{B}:= \begin{bmatrix} 1 & 0 \\ 0 & 3\end{bmatrix}
\otimes \begin{bmatrix} 0 & \tilde{x} & \tilde{y} \\ 
-\tilde{x} & 0 & \tilde{z}\\ -\tilde{y} & -\tilde{z} & 0\end{bmatrix}.
\end{equation}
Thus, it is clear that $H_{(1,-1)}$ and $H_{(1,1)}$ are isomorphic to $Or_p$ and furthermore,
$R_p=H_{(1,-1)}H_{(1,1)}\cong Or_p^{\circ (1,3)}$.  
If $3\equiv \alpha^{-2}~ (p)$, for some $\alpha\in \mathbb{Z}$, then set:
\begin{equation}
	H_{(\alpha,\alpha)}:=\langle a^{\alpha}d^{\alpha}, 
		b^{\alpha}e^{\alpha}, c^{\alpha} f^{\alpha}\rangle.
\end{equation}
Thus, $R_p=H_{(1,-1)}H_{(\alpha,\alpha)}\cong Or_p^{\circ(1,1)}$.

$(ii)$.  If $p=3$ we can compute $Z(R_3)$ directly to verify the 
properties.  However, an alternative approach is to use the related
bilinear map $\Bi(R_3,R_3'):R_3/R_3'\times R_3/R_3'\to R_3'$.  This
produces a $\mathbb{Z}_3$-bilinear map exactly as in \eqref{eq:B2}.
The only exception is that $B$ is degenerate.  Computing a basis for
the radical of $B$ can be done by computing a basis for the radical of 
$D=\begin{bmatrix} 2 & 1 \\ 1 & 2\end{bmatrix}$; e.g.: 
write $D$ with respect to the basis $\{(1,0),(1,1)\}$.  
Pulling back this basis to $R_3/R_3'$ we have $b(u,v)=u\tilde{B}v^t$, 
for all $u,v\in \mathbb{Z}_3^6$, where
\begin{equation}
\tilde{B} := \begin{bmatrix} 2 & 0 \\ 0 & 0 \end{bmatrix}
\otimes \begin{bmatrix}
0 & x & y \\ -x & 0 & z\\ -y & -z & 0 \end{bmatrix}.
\end{equation}
Pulling back to subgroups of $R_p$ we have the following
central factors:
\begin{eqnarray}
	H_{(1,0)}  :=  \langle a,b,c\rangle\cong Or_3,\textnormal{ and }
	H_{(1,1)} :=\langle ad,be,cf\rangle R_3'\cong \mathbb{Z}_3^6.
\end{eqnarray}
So $R_3=H_{(1,0)}H_{(1,1)}\cong Or_3\times \mathbb{Z}_3^3$.
\end{proof}

We now construct the example of C.Y. Tang 
\cite[Section 6]{Tang:cent-2} of a $2$-group with fully
refined central decompositions of different sizes.  This
demonstrates where the algorithm for \thmref{thm:main} 
must proper select a central decomposition of maximum possible length.
\begin{ex}[C.Y. Tang]\label{ex:Tang}
\begin{equation}
R_2\circ Or_2:=\langle a,b,c,d,e,f\rangle\times\langle s,t,u\rangle
/\langle [a,e][s,t]^{-1},[a,f][s,u]^{-1},[b,f][t,u]^{-1}\rangle
\end{equation}
is isomorphic to $Or_2\circ Or_2\circ Or_2$.  Yet $R_2$ and $Or_2$ are
both centrally indecomposable.
\end{ex}
We provide an alternative proof using the approach of \thmref{thm:main}.
\begin{proof}  
Let $P:=R_2\circ Or_2$.  Using the obvious bases of $R_2/Z(R_2)\times
Or_2/Z(Or_2)$ given by \eqref{eq:odd-even-ex} and
\eqref{eq:free-rank-3}, produces the following matrix defining $\Bi(P)$:
\begin{equation}
B:= D\otimes \begin{bmatrix} 0 & x & y\\ -x & 0 & z\\ -y & -z & 0 
\end{bmatrix},
\qquad 
D:=\begin{bmatrix} 0 & 1 & 0 \\ 1 & 0 & 0 \\ 0 & 0 & 1\end{bmatrix}.
\end{equation}
Thus, $\Adj(B)\cong \Adj(D)$.  The map $d:\mathbb{Z}_2^3\times \mathbb{Z}_2^3\to \mathbb{Z}_2$ defined by $d(u,v)=uDv^t$ is a nondegenerate 
symmetric bilinear form which has an orthonormal basis
$\{(0,1,1),(1,1,1),(1,0,1)\}$.  Evidently this produces a 
$\perp$-decomposition of $d$ (and $b$) of maximum possible
length. The corresponding fully refined central decomposition of $P$ 
has the following factors:
\begin{eqnarray}
	H_{(0,1,1)} & := & \langle ds,et,fu\rangle\cong Or_2,\\
	H_{(1,1,1)} & := & \langle ads,bet,cfu\rangle\cong Or_2,
		\textnormal{ and }\\
	H_{(1,0,1)} & := & \langle as,bt,cu\rangle \cong Or_2.
\end{eqnarray}
\end{proof}

\begin{remark}
Our proof of Example \ref{ex:Tang} can be applied to 
central products where $Or_2$ is replaced by any $2$-group of 
orthogonal type.  Asymptotically, there are $2^{2n^3/27+O(n^2)}$ such groups of order $2^n$ \cite{Wilson:indecomp}; thus, there are infinite expanding families of examples of the type introduced by Tang.
\end{remark}

%
%

\def\cprime{$'$}

\end{document}